\documentclass[a4paper,12pt,reqno]{article}
 \usepackage[english]{babel}
  \usepackage{amsmath,amsfonts,amssymb,amsthm,bbm}
   \usepackage{graphics,epsfig,psfrag} %%add this and next lines if pictures
   \usepackage{paralist,scrextend,subfigure}
   \usepackage[dvipsnames]{xcolor}
  \usepackage{bbm}
  \usepackage{stmaryrd} 	
   \usepackage{hyperref} %%Warning: when you first run your tex
\usepackage{float}

\usepackage{esint} %%% Pour tracer des intégrales barrées

    \usepackage[active]{srcltx} %%allows you to jump from your editor
    \usepackage{color,soul} %%allows you to use the \st{...} command, which
    \usepackage[latin1]{inputenc}
    \usepackage[OT1]{fontenc}
    \usepackage{authblk}

\begin{document}

\theoremstyle{plain}

\newtheorem{theorem}{Theorem}
\newtheorem{lemma}{Lemma}
\newtheorem{cor}{Corollary}
\newtheorem{prop}{Proposition}
\newtheorem{req}{Remark}
\newtheorem{definition}{Definition}
\newtheorem{hyp}{Hypothesis}
\newtheorem{question}{Question}
\newtheorem{open question}{Open question}
\newtheorem{conjecture}{Conjecture}

\setlength{\marginparwidth}{1in}

\newenvironment{Romain}
    {\begin{center}
     \begin{tabular}{|p{0.95\textwidth}|}\hline\color{ForestGreen}}
    { \vspace{5pt}\color{black}\\\hline\end{tabular} 
    \end{center}}

\title{Propagation properties of reaction-diffusion equations in periodic domains}
\makeatletter{\renewcommand*{\@makefnmark}{}\footnotetext{email : romain.ducasse@dauphine.fr}\makeatother}

\author[*]{Romain {\sc Ducasse}}

\affil[*]{Ecole des Hautes Etudes en Sciences Sociales, 
		 PSL Research University,  
		Centre d'Analyse et Math\'ematiques Sociales,
		 54 boulevard Raspail 
		75006 Paris, France}
		
\date{}

\maketitle

\noindent {\textbf{Keywords:} }Propagation, spreading, reaction-diffusion equations, heat kernel, domains with holes.
\\
\noindent {\textbf{MSC:} } 35A08, 35B30, 35K05, 35K57, 35B40

\begin{abstract}
This paper studies the phenomenon of \emph{invasion} for heterogeneous reaction-diffusion equations in periodic domains with monostable and combustion reaction terms. We give an answer to a question rised by Berestycki, Hamel and Nadirashvili in \cite{BHN1} concerning the connection between the speed of invasion and the speed of fronts. To do so, we extend the classical Freidlin-Gartner formula to such equations, using a geometrical argument devised by Rossi in \cite{R1}, and derive some bounds on the speed of fronts using estimates on the heat kernel.
\end{abstract}

\section{Introduction and results}

\subsection{Introduction}

This paper deals with the spreading properties of the following reaction-diffusion equation:
\begin{equation}\label{eqgen}
\left\{
\begin{array}{rrll}
\partial_{t}u &= &\text{div} (A(x)\nabla u ) + q(x)\cdot\nabla u +f(x,u), &\quad t>0, \ x\in \Omega, \\
\nu \cdot A(x)\nabla u &=& 0,  &\quad t>0, \ x\in \partial \Omega.
\end{array}
\right.
\end{equation}
In the whole paper, the domain $\Omega$ and the coefficients are assumed to be periodic. Here, $\nu$ stands for the exterior normal. Reaction-diffusion equations arise in the study of various phenomena in biology (propagation of genes, epidemics), physics (combustion), and more recently in social sciences (rioting models). A particular emphasis is given here to the case where the equation is homogeneous but the domain is not the whole space:

\begin{equation*}
\left\{
\begin{array}{rrll}
\partial_{t}u &= &\Delta u  +f(u), &\quad t>0 , \ x\in \Omega, \\
\partial_{\nu} u &=& 0,  &\quad t>0 , \ x\in \partial \Omega.
\end{array}
\right.
\end{equation*}
In such case, we provide an answer to a question asked by Berestycki, Hamel and Nadirashvili in \cite{BHN1} concerning the relation between the speed of invasion and the speed of fronts for this problem.

Reaction-diffusion equations were extensively studied since the seminal paper of Kolmogorov, Petrovski and Piskunov \cite{KPP}. There, the authors dealt the homogeneous equation

\begin{equation}\label{KPP}
\partial_{t}u = \Delta u + f(u), \quad t>0 ,\ x\in \mathbb{R}^{N},
\end{equation}
with $f(u) = u(1-u)$.
The results of \cite{KPP} have been extended by Aronson and Weinberger in \cite{AW} to more general \emph{reaction terms} $f$. The basic assumption is that $f(0)=f(1)=0$, so that the constant states $u\equiv 0$ and $u\equiv 1$ are stationary solutions. We shall pay a particular attention to the two following types of nonlinearities:
\begin{labeling}{\quad monostable\quad}
  \item[\quad {\em monostable}\quad] 
  $f>0\quad\text{in }(0,1)$;
  
  \item[\quad {\em combustion}] 
  $\exists\theta\in(0,1),\quad 
f=0\quad\text{in }[0,\theta],\quad f>0\quad\text{in }(\theta,1)$.

\end{labeling}
 These two notions extend to the case where $f$ can depend on $x$, see Definition \ref{def f} below. Two important features of reaction-diffusion equations have been derived in \cite{AW}. First, equation \eqref{KPP} admits particular solutions called \emph{traveling fronts}. These are positive entire (i.e., defined for all $t\in \mathbb{R}$) solutions of the form $u(t,x) = \phi ( x\cdot e -ct)$, for some $e\in \mathbb{S}^{N-1}$, $c \in \mathbb{R}$, $\phi$ decreasing and satisfying $\phi(s) \to 1$ as $s \to -\infty$ and $\phi(s) \to 0$ as $s \to +\infty$. The unit vector $e$ is the \emph{direction of propagation}, $c$ is the \emph{speed of propagation} and $\phi$ is the \emph{profile} of the traveling front. More specifically, there exists a quantity $c^{\star}$ such that there are fronts with speed $c$, for every $c \geq c^{\star}$ if $f$ is of the \emph{monostable} type, whereas  there are traveling fronts only with speed $c =c^{\star}$ if $f$ is of the \emph{combustion} type. Of course, the homogeneity of equation \eqref{KPP} implies that the quantity $c^{\star}$ does not depend on the direction of the fronts $e$. We mention that, if $f$ is of the \emph{KPP} type (i.e., if it is monostable and satisfies $f^{\prime}(0)>0$ and $f(u)\leq f^{\prime}(0)u$, for $u \in [0,1]$), then it is proven in \cite{KPP} that $c^{\star} = 2\sqrt{f^{\prime}(0)}$. The quantity $c^{\star}$ is called the \emph{critical (or minimal) speed of fronts}. We consider this quantity in a more general context in Section \ref{sectionC}.

The second important feature is the property of \emph{invasion}. If $u(t,x)$ is the solution of \eqref{KPP} emerging from a non-negative compactly supported initial datum $u_{0}$, does it converge to $1$ as $t$ goes to $+\infty$ ? If this convergence holds (locally uniformly in $x$), we say that invasion occurs for the initial datum $u_{0}$. Of course, this depends on the nonlinearity $f$. For instance, if $f$ is of the combustion type, and if the compactly supported non-negative initial datum verifies $u_{0} \leq \theta$, then the problem \eqref{KPP} boils down to the heat equation, and then $u(t,x) \to 0$ as $t\to +\infty$ uniformly in $x$. However, it is shown in \cite{AW} that, for every $\eta \in (\theta , 1)$, there is $R>0$ such that any initial datum such that $u_{0}(x)\geq \eta \mathbbm{1}_{B_{R}}$ (where $B_{R}$ is the ball of center $0$ and of radius $R$) satisfies the invasion property. In contrast, if $f$ is of the KPP type, then invasion occurs for any non-negative non-zero initial datum.

Once we know that invasion occurs for some initial data, we can define the \emph{speed of invasion}. We say that $w(e) > 0$ is the speed of invasion for \eqref{KPP} in the direction $e \in \mathbb{S}^{N-1}$ if, for any solution $u(t,x)$ of \eqref{KPP}  emerging from a compactly supported non-negative initial datum which converges to $1$ as $t$ goes to $+\infty$, locally uniformly in $x$, the following holds:

\begin{equation*}
\begin{array}{llc}
\forall c>w(e), \quad &u(t,x+cte) \to 0 \quad &\text{as} \quad t \to +\infty, \\
\forall  c\in [0,w(e)), \quad &u(t,x+cte) \to 1 \quad &\text{as} \quad t \to +\infty,
\end{array}
\end{equation*}
locally uniformly in $x\in \mathbb{R}^{N}$. Again, the homogeneity of equation \eqref{KPP} yields that the speed of invasion is actually independent of the direction $e$. Moreover, if $f$ is of the KPP type, it is proven in \cite{KPP} that $w(e) = 2\sqrt{f^{\prime}(0)}, \ \forall e \in \mathbb{S}^{N-1}$. Hence, in this case $c^{\star} \equiv w$.

One of the main motivation behind the present paper is to understand the connections between the speed of fronts $c^{\star}$ and the speed of invasion $w$ in a more general context. In order to state our main results, we first present how the notions of fronts and invasion extend to the case of spatially periodic heterogeneous equations.

\subsection{Pulsating traveling fronts}\label{sectionC}

Berestycki and Hamel extend in \cite{BH1} the notion of traveling fronts to the more general framework of equation \eqref{eqgen}. Throughout the whole paper, we assume that $A,q,f,\Omega$ are periodic, with the same period, i.e, there are $L_1, \dots , L_N >0$ such that 
$$
\forall k \in \prod_{i=1}^{N}L_i \mathbb{Z}, \quad \Omega + \left\{k\right\} = \Omega,
$$
and 
$$
\forall k \in \prod_{i=1}^{N}L_i \mathbb{Z},\quad  f(\cdot+k,\cdot) = f,\quad q(\cdot+k)=q,\quad A(\cdot+k) = A.
$$
We shall denote $\mathcal{C} := \prod_{i=1}^{N}[0,L_i)$ the periodicity cell. Typical examples of such domains $\Omega$ are domains with ``holes" : if $K\subset \mathbb{R}^{N}$ is a smooth compact set, we can define the periodic domain $\Omega := \left(K +L \mathbb{Z}^{N}\right)^{c}$, with $L>0$ large enough so that the resulting domain is smooth and connected. This domain can be seen as the whole space with $K$-shaped ``holes" periodically  distributed.

To simplify the notations, unless otherwise stated, we shall assume that the period is $1$, i.e., $L_1=\ldots=L_N =1$. In order to apply the results of  \cite{BH1}, we make the following assumptions on the domain:

\begin{equation}\label{regomega}
\Omega \text{ is a periodic, connected open subset of } \mathbb{R}^{N} \text{of class } C^{3},
\end{equation}
and the following hypotheses on the coefficients:
\begin{equation}\label{regularity}
\left\{
\begin{array}{llc}
A \in C^{3}(\overline{\Omega}) \text{ is symmetric and uniformly elliptic and periodic},\\
q\in C^{1,\alpha}(\overline{\Omega}) \text{ for some } \alpha \in (0,1), \ \text{div }q =0 , \ \int_{\mathcal{C}\cap\Omega } q =0,\ q \text{ is periodic}, \\
f : \overline{\Omega}\times \left[ 0,1\right] \mapsto \mathbb{R} \text{ is of class } C^{1,\alpha} \text{ for some }\alpha \in (0,1).
\end{array}
\right.
\end{equation}
We also assume that the nonlinearity $f$ satisfies the following

\begin{equation}\label{hypf}
\left\{
\begin{array}{llc}
\forall x \in \Omega, \quad f(x,0)=f(x,1)=0, \\
\exists S \in (0,1), \ \forall x \in \overline{\Omega}, \quad f(x,\cdot) \text{ is nonincreasing in } \left[ S,1 \right], \\ 
\forall s \in (0,1) , \quad f(\cdot,s) \text{ is periodic}.
\end{array}
\right.
\end{equation}

By analogy with the homogeneous case $f=f(u)$, we define monostable, KPP and combustion nonlinearities $f(x,u)$:

\begin{definition}\label{def f}
We say that $f$ is of the \emph{monostable} type if
\begin{equation}\label{monostable}
\forall s \in (0,1),\quad \min_{x\in \overline{\Omega}} f(x,s)\geq 0, \quad \max_{x\in \overline{\Omega}}f(x,s)>0.
\end{equation}
Among monostable nonlinearities, there is the special class of \emph{KPP} nonlinearities. In addition to being monostable, they satisfy

\begin{equation}\label{defkpp}
\forall x \in \overline{\Omega},\ \forall s \in [0,1], \quad f(x,s) \leq \partial_{s}f(x,0)s.
\end{equation}
We say that $f$ is of the \emph{combustion} type if
\begin{equation}\label{combustion}
\left\{
\begin{array}{lc}
\exists \theta \in (0,1), \ \forall (x,s) \in \Omega \times \left[ 0,\theta \right], \ f(x,s)=0, \\
\forall s \in (\theta,1), \ \min_{x\in \overline{\Omega}}f(x,s)\geq 0, \  \max_{x\in \overline{\Omega}}f(x,s)>0.
\end{array}
\right.
\end{equation}
 
\end{definition}
The important difference between combustion and monostable nonlinearities (from which stems the non-uniqueness of speeds of fronts for monostable equation) is that, when $f$ is of the combustion type,

\begin{equation}\label{decroissance}
\exists \theta \in (0,1],\ \forall x \in \Omega, \quad f(x,\cdot) \text{ is nonincreasing in } \left[ 0,\theta \right].
\end{equation}

In the periodic framework, the notion of traveling fronts can be generalized by \emph{pulsating traveling fronts}.

\begin{definition}\label{defPTF}
A pulsating traveling front in the direction $e \in \mathbb{S}^{N-1}$ of speed $c \in \mathbb{R}\backslash\{0\}$ connecting $1$ to $0$ is an entire (i.e., defined for all $t\in \mathbb{R}$) solution $v$ of \eqref{eqgen} satisfying

\begin{equation*}
\left\{
\begin{array}{llc}
\forall k \in \mathbb{Z}^{N}, \ \forall x \in \Omega, \quad v(t+\frac{k\cdot e}{c},x) = v(t,x-k), \\
v(t,x) \to 1 \, \text{ as } x \cdot e \to -\infty, \quad v(t,x) \to 0 \, \text{ as } x\cdot e \to +\infty. \\
\end{array}
\right.
\end{equation*}
\end{definition}
Such fronts are known to exist in several situations. For instance, it is proven in \cite{BH1} that, under hypotheses \eqref{regularity}-\eqref{hypf}, for every $e\in \mathbb{S}^{N-1}$, there is $c^{\star}(e)>0$, called the \emph{critical (or minimal) speed of fronts} in direction $e$, such that pulsating traveling fronts in the direction $e$ with speed $c$ exists if, and only if, $c\geq c^{\star}(e)$ when $f$ is of the monostable type \eqref{monostable} or  only if $c=  c^{\star}(e)$ when $f$ is of the combustion type \eqref{combustion}, see   \cite[Theorems 1.13 - 1.14]{BH1}.

\subsection{The speed of invasion}

The results of Kolmogorov, Petrovski, Piskunov \cite{KPP} and Aronson and Weinberger \cite{AW} concerning the invasion have also been extended to a more general framework than the homogeneous one. First, consider the equation

\begin{equation}\label{eqFG}
\partial_{t}u = \text{div} (A(x)\nabla u ) + q(x)\cdot\nabla u +f(x,u), \quad t>0 , \ x\in \mathbb{R}^{N}.
\end{equation}
Then, one can define the speed of invasion $w$ as a function from the unit sphere $\mathbb{S}^{N-1}$ to $\mathbb{R}^{+}$ such that, for every $u$ solution of \eqref{eqgen} arising from a compactly supported non-negative initial datum which converges to $1$ as $t$ goes to $+\infty$, locally uniformly in $x \in\mathbb{R}^{N}$, we have, for $e\in \mathbb{S}^{N-1}$:

\begin{equation*}
\begin{array}{llc}
\forall c>w(e), \quad &u(t,x+cte) \to 0 \quad &\text{as} \quad t \to +\infty, \\
\forall  c\in [0,w(e)), \quad &u(t,x+cte) \to 1 \quad &\text{as} \quad t \to +\infty,
\end{array}
\end{equation*}
locally uniformly in $x\in \mathbb{R}^{N}$.

Using probabilistic techniques, Freidlin and Gartner show in \cite{FG} the existence of a speed of invasion for the equation \eqref{eqFG} when $f$ is of the KPP type \eqref{defkpp} and $A,q,f$ are $x-$periodic. They show that invasion occurs for every non-negative non-null compactly supported initial datum and prove what is now known as the \emph{Freidlin-Gartner formula} :

\begin{equation}\label{FGoriginal}
w(e):=\min_{\substack{\xi \in \mathbb{R}^{N} \\ e \cdot \xi >0}} \frac{k(\xi)}{e\cdot \xi},
\end{equation}
where $k(\xi)$ is the periodic principal eigenvalue of the operator
$$
L_{\xi}u := \text{div}(A\nabla u)-2\xi\cdot A\nabla u +q\cdot \xi u +(-\text{div}(A\xi)-q\cdot \xi +\xi\cdot A \xi + \partial_{u}f(x,0))u.
$$

This formula is also proved by Berestycki, Hamel and Nadin in \cite{BHN} using a PDE approach. Similar properties of spreading for heterogeneous reaction-diffusion equations are studied with other approaches : the viscosity solution/singular perturbation method is adopted by Evans and Souganidis in \cite{ES} and Barles, Soner and Souganidis in \cite{BSS}. Weinberger uses an abstract discrete system approach in \cite{W}.

Berestycki, Hamel and Nadirashvili show in \cite{BHN1} that, if one considers KPP nonlinearities, the quantity $c^{\star}(e) := \min_{\lambda>0}\frac{k(\lambda e)}{\lambda}$ (where $k$ is the principal eigenvalue introduced before and $e\in \mathbb{S}^{N-1}$) coincides with the critical speed of \emph{pulsating traveling fronts} in the direction $e$ for equation \eqref{eqFG} (if the equation were set on a periodic domain $\Omega$ instead of $\mathbb{R}^{N}$, this relation still holds true with $k$ being the periodic principal eigenvalue of the same operator but with the additional boundary condition $\nu \cdot A\nabla u = \lambda (\nu \cdot e) u$ on $\partial \Omega$, see \cite{BHN1} for the details). 
Consequently, in the KPP case, the Freidlin-Gartner formula \eqref{FGoriginal} can be rewritten as
\begin{equation}\label{FGformula}
w(e) = \min_{e\cdot \xi >0}\frac{c^{\star}(\xi)}{e\cdot \xi}.
\end{equation}
The fact that pulsating traveling fronts exist not only in the KPP case but also for other reaction terms, and hence that the formula \eqref{FGformula} could make sense in more general frameworks than the KPP one, led Rossi to extend in \cite{R1} the Freidlin-Gartner formula to much more general equations in the whole space (essentially, all those for which pulsating traveling fronts are known to exist).

In this paper, we deal with invasion in domains $\Omega$ that are not necessarily $\mathbb{R}^{N}$. In this case, it is convenient to introduce the  notion of \emph{asymptotic set of spreading}.

\begin{definition}\label{as}
Let $\mathcal{W}\subset \mathbb{R}^{N}$ be a closed set coinciding with the closure of its interior. We say that $\mathcal{W}$ is the asymptotic set of spreading for a reaction-diffusion equation if, for any bounded solution $u(t,x)$ emerging from a non-negative compactly supported initial datum such that $u(t,x)\to1$ as $t\to +\infty$, locally uniformly in $x \in \overline{\Omega}$, we have:
\begin{equation}\label{subset}
\forall K \text{ compact}, \ K\subset int(\mathcal{W}), \quad \inf_{x\in tK} u(t,x) \to 1 \quad \text{as } \, t\to +\infty,
\end{equation}
\begin{equation}\label{superset}
\forall C \text{ closed}, \ C \cap \mathcal{W} = \emptyset, \quad \sup_{x\in tC} u(t,x) \to 0 \quad \text{as } \, t\to +\infty.
\end{equation}
If only \eqref{subset}, respectively \eqref{superset}, holds, $\mathcal{W}$ is said to be an asymptotic subset, respectively superset, of spreading.
\end{definition}
The asymptotic set of spreading relates to the notion of speed of invasion previously described. Indeed, assume that $\mathcal{W}$ is an asymptotic set of spreading and that we can write $\mathcal{W} = \left\{ r \xi \, : \, \xi \in \mathbb{S}^{N-1} , \, 0\leq r \leq w(\xi) \right\}$ with $w$ a continuous function. Then, if $\Omega =  \mathbb{R}^{N}$, $w(e)$ is the speed of spreading in the direction $e$, as defined before. For example, if $f$ is a KPP nonlinearity independent of $x$, then the asymptotic set of spreading associated with the homogeneous equation \eqref{KPP} is the ball of center $0$ and of radius $2\sqrt{f^{\prime}(0)}$. We emphasize that this is somewhat stronger than saying that the invasion speed is $2\sqrt{f^{\prime}(0)}$.

Observe that, for the definition of the asymptotic set of spreading to be meaningful, it is necessary that there are compactly supported initial datum $u_{0}$ for which the invasion property holds. Rossi and the author give in \cite{DR} sufficient conditions to have invasion for equation \eqref{eqgen}. In particular, we show there that, if $f$ is of the monostable or combustion type, in the sense of Definition \ref{def f}, and if the drift term $q$ is ``not too large" (see \cite{DR} for the details), then, denoting
$$
\theta := \max \left\{ s \in [0,1)  \ :\  \exists x\in\overline{\Omega},\
	f(x,s)=0    \right\},
$$
we have that, for all $\eta\in(\theta,1)$, there is $r>0$ such that any solution of \eqref{eqgen} with an initial datum $u_0$ satisfying 
$$
u_0>\eta\ \text{ in }\Omega\cap B_r,
$$
converges to $1$ as $t$ goes to $+\infty$, locally uniformly in $x \in \overline{\Omega}$.

\subsection{Statement of the main results}

One of the main motivation behind the present paper is to answer the following question, raised by Berestycki, Hamel and Nadirashvili in \cite{BHN1}: 

\begin{question}\label{question}
Consider the homogeneous equation set on a periodic domain $\Omega$

\begin{equation}\label{PeriodicFKPP}
\left\{
\begin{array}{rrll}
\partial_{t}u - \Delta u &=&  f(u), &\quad t>0 , \ x\in \Omega, \\
\partial_{\nu} u &=& 0,  &\quad t>0 ,\ x\in \partial \Omega.
\end{array}
\right.
\end{equation}
Are there domains $\Omega$ such that $c^{\star} \not\equiv w$ ?
\end{question}
We recall that $c^{\star}$ is the critical speed of pulsating traveling fronts and $w$ is the speed of invasion. Originally, this question was asked for $f$ of KPP type \eqref{defkpp}, but it also makes sense if $f$ is a monostable \eqref{monostable} or a combustion \eqref{combustion} nonlinearity.

The speed of invasion and the critical speed of fronts for the homogeneous problem \eqref{KPP} in the whole space are actually independent of the direction thanks to the homogeneity of the equation and of the domain, and are everywhere equal : $c^{\star} \equiv w$. In the KPP case, as we mentioned earlier, both are equal to $2\sqrt{f^{\prime}(0)}$. 
If the domain $\Omega$ is periodic, the speeds $c^{\star}(e)$ and $w(e)$ can depend on the direction $e$.

We mention that, if one considers the equation with general coefficients \eqref{eqgen}, it is possible to have $c^{\star}\not\equiv w$ even if $\Omega = \mathbb{R}^{N}$. The first example one can check is, in dimension $2$, when the Laplace operator is replaced by $a\partial_{xx}^{2} + b\partial_{yy}^{2}$ with $a,b>0$. In this case, the exact values of $c^{\star}$ and $w$ can be computed, see \cite[Remark 1.12]{BHN1}, and one can observe that, if $a\neq b$, then $c^{\star}\not\equiv w$. This fact can also be deduced from our Proposition \ref{prop} below.  What was not known is whether the geometry of the domain alone could entail that $c^{\star}\not\equiv w$. We prove that this is the case.

\begin{theorem}\label{thwc}

Let $f$ be a monostable \eqref{monostable} or a combustion \eqref{combustion} nonlinearity independent of $x$. There are smooth periodic domains $\Omega$ such that the critical speed of pulsating traveling fronts $c^{\star}$ and the invasion speed $w$ for equation \eqref{PeriodicFKPP} do not coincide in every direction.

\end{theorem}
This provides a positive answer to Question \ref{question}.  When the nonlinearity $f$ is of the monostable or combustion type, then the domains we exhibit are $L$-periodic, with $L$ large enough. If $f$ is a KPP nonlinearity, then we can build domains with any periodicity.

As a preliminary step to prove Theorem \ref{thwc}, we show that the Freidlin-Gartner formula \eqref{FGformula} holds true for the general equation \eqref{eqgen} in the periodic domain $\Omega$, extending then Rossi's result to the case where the domain is not $\mathbb{R}^{N}$ anymore. This is done in Section \ref{sectionFG}. More precisely, we prove the following:

\begin{theorem}\label{FG}
Let $A , q ,f ,\Omega$ be periodic, satisfying \eqref{regularity}-\eqref{hypf}. Assume that $f$ is a monostable \eqref{monostable} or a combustion \eqref{combustion} nonlinearity. Then, equation \eqref{eqgen} has the following asymptotic set of spreading:
\begin{equation}\label{W}
\mathcal{W} = \left\{ r \xi \, : \, \xi \in \mathbb{S}^{N-1} , \, 0\leq r \leq w(\xi) \right\},
\end{equation}
where $w(\xi):= \inf_{e\cdot \xi >0}\frac{c^{\star}(e)}{e\cdot \xi}$, and $c^{\star}(e)$ is the critical speed of pulsating traveling fronts in the direction $e$.
\end{theorem}
Once  Theorem \ref{FG} established, we employ it to derive a simple criterion ensuring that $c^{\star}\not\equiv w$ : we show that, if $c^{\star}\equiv w$, then $c^{\star}$ and $w$ are necessarily constant, see Proposition \ref{prop}. To anwer Question \ref{question} then tantamount to finding domains where $w$ or $c^{\star}$ are not constant. Intuitively, we may think that, if a domain is ``very obstructed" in a direction, then the speed should be small in this direction.

In order to make this intuition rigorous, we derive new estimates on the invasion speed that do take into account the geometry of the domain. This is the subject of Section \ref{geodesic estimates}. The main tool is an upper bound on the heat kernel in $\Omega$. 
Once these estimates at hand, we are able to build domains where $c^{\star}$ and $w$ are not constant, hence different. This is done in Section \ref{c et w differ}, proving then Theorem \ref{thwc}.

Observe that, though we can build domains $\Omega$ where $c^{\star}\not\equiv w$, there is at least one direction $e\in \mathbb{S}^{N-1}$ such that $c^{\star}(e)=w(e)$. Indeed, the Freidlin-Gartner formula \eqref{FGformula} given by Theorem \ref{FG} easily implies that, if $e_{min}$ is a direction that minimizes $c^{\star}$ (which is known to be a lower semi-continuous function, see \cite{AG, R1}), then $c^{\star}(e_{min}) = w(e_{min}) $. The only other characterization of directions $e$ where $c^{\star}(e) =w(e)$ we are aware of is in the KPP case :  it is proven in \cite{BHN1} that  $c^{\star}(e_{inv}) = w(e_{inv})$ if $\Omega$ is invariant in the direction $e_{inv}$ ( i.e.,  $\Omega + \left\{\lambda e_{inv}\right\} = \Omega$, for all $\lambda \in \mathbb{R}$), see \cite{BHN1}.

We show in Section \ref{equality} that, if the domain $\Omega$ satisfies some geometrical hypotheses and if $u \mapsto \frac{f(u)}{u}$ is non-increasing (which implies in particular that $f$ is KPP), then there are directions that satisfy the equality between $c^{\star}$ and $w$. More specifically, we derive the following:

\begin{theorem}\label{symmetry}
Assume that $f$ satisfies \eqref{regularity}, \eqref{hypf} and $u \mapsto \frac{f(u)}{u}$ is non-increasing. Let $c^{\star}$ and $w$ be the critical speed of fronts and the speed of invasion for equation \eqref{PeriodicFKPP}. Then, 
$$
c^{\star}(e) = w(e)
$$ 
in the following cases:
\begin{enumerate}[(i)]

\item If $N = 2$ and $\Omega$ is invariant with respect to the symmetry about the axis directed by $e$, i.e., denoting $S$ such a symmetry, if we have:

$$
 S\Omega = \Omega.
$$

\item If $N \geq 3$ and $\Omega$ is invariant with respect to the rotation around the axis directed by $e$ of angle $\pi$, i.e., denoting $R$ such a rotation, if we have:

$$
R\Omega = \Omega.
$$

\end{enumerate}
\end{theorem}

Let us conclude this section with some questions that are still open. The set $\mathcal{W}$ given by \eqref{W} is sometimes called the \emph{Wulff shape} associated with the surface tension $c^{\star}$. It appears in crystallography and in isoperimetric problems. A natural question is whether the function $w$ parametrizing the boundary of $\mathcal{W}$ is regular. Rossi proves in \cite{R1} that it is continuous. We are not aware of further regularity results. We conjecture that, if $c^{\star}$ is smooth (which is the case if $f$ is of the KPP type), then, $w$ is smooth.

Theorem \ref{thwc} states that there are domains $\Omega$ such that $c^{\star}\not\equiv w$. One may wonder on the contrary if there are periodic domains $\Omega \neq \mathbb{R}^{N}$, such that $c^{\star} \equiv w$. Thanks to our Proposition \ref{prop}, this is equivalent to finding domains where $c^{\star}$ is constant. As far as we know, the existence of such domains is still open

\begin{req} 
In addition to the monostable and combustion cases, there is another class of reaction terms $f$  that is widely studied in the literature, namely the \emph{bistable nonlinearities}. The prototype is $f(u) = u(1-u)(u-a)$, with $a\in (0,1)$. In this paper, we \emph{do not} consider such nonlinearities : indeed, the main tool we use is the existence of pulsating traveling fronts with \emph{positive speed}. If $\Omega = \mathbb{R}^{N}$, there are results in some particular cases, see \cite{Ducrot,Xin1, Xin2} for instance. If $\Omega \neq \mathbb{R}^{N}$, the situation is yet to explore, and the geometry of the domain can yield phenomena that do not appear in the combustion or monostable case. For instance, Rossi and the author show in \cite{DR} that invasion can occur in some direction but not in others. However, we mention that the strategy used to derive Theorem \ref{thwc} above still applies if, for every $e\in \mathbb{S}^{N}$, there are pulsating traveling fronts with positive speed, even is $f$ is bistable. 
\end{req}

\section{Freidlin-Gartner formula for a periodic domain}\label{sectionFG}

This section is dedicated to the proof of Theorem \ref{FG}, i.e., we show that the Freidlin-Gartner formula \eqref{FGformula} relating the speeds of fronts to the speed of invasion still holds true when the domain is not $\mathbb{R}^{N}$ but a periodic domain $\Omega$ and with monostable or combustion nonlinearities. Our proof is based on the same strategy as the one used by Rossi in \cite{R1}. We start to state some preliminary technical results. For simplicity, we assume in this whole section that the domain and the coefficients are $1$-periodic, i.e., $L_1 = \ldots = L_N =1$.

\subsection{Preliminary results}

In the proof of Theorem \ref{FG}, we will need some technical lemmas. They generalize those of \cite[Section 2.1]{R1} to the case where the domain is not $\mathbb{R}^{N}$ anymore. The main technical difficulty is that $\Omega$ is not invariant under translations in general. The proofs follow the same lines as in \cite{R1}, and can be found for completeness in the Appendix. We say that $u$ is a subsolution (respectively supersolution) if it satisfies \eqref{eqgen} with the the symbols $=$ replaced by $\leq$ (respectively $\geq$).

The first lemma states that, every entire solution that is ``large enough" in some direction is actually ``front-like" in this direction.
\begin{lemma}\label{lemma1}
Let $\gamma > 0$. Assume that \eqref{regularity} and \eqref{hypf} hold. Let ${\displaystyle {u\in C^{1+\alpha/2,2+\alpha}\big(\mathbb{R}\times\Omega\big)}}$ for some $\alpha \in (0,1)$ be an entire supersolution of \eqref{eqgen} such that

$$
\inf_{\substack{t<0\\ x \cdot e<\gamma t \\ x\in\Omega }} u(t,x) \: > S,
$$
where $S$ is defined in \eqref{hypf}. Then:

$$
\liminf_{\delta\to+\infty}\inf_{\substack{t<0  \\  x\cdot e< \gamma t -\delta  \\ x\in\Omega }} u(t,x) \geq 1.
$$

\end{lemma}
The following lemma is a comparison principle for front-like solutions.

\begin{lemma}\label{lemma2}

Assume that \eqref{regularity} and  \eqref{hypf}  hold. Let $\overline{u}, \ \underline{u} \in C^{1+\alpha/2,2+\alpha}\big(\mathbb{R}\times\Omega\big)$, for some $\alpha\in(0,1)$, be respectively entire supersolution and subsolution of \eqref{eqgen}. Assume that there are $e\in\mathbb{S}^{N-1}$, $\gamma > 0 $ such that

\begin{equation}\label{compsursol}
\overline{u}>0,\ \ \liminf_{\delta\to+\infty}\inf_{\substack{t<0\\
x\cdot e<\gamma t - \delta\\
x\in\Omega }}\overline{u}(t,x)\:\geq 1.
\end{equation}
Moreover, assume that $\underline{u}\leq1$ and that there is $\eta>0$
such that the following hold:
\begin{itemize}

\item The nonlinearity $f$ is of the combustion type \eqref{combustion} and we have:
\begin{equation}\label{compcomb}
\forall s>0,\ \exists L\in\mathbb{R},\ \underline{u}(t,x)\leq s\ \text{ if } \ t\leq0,\ x\cdot e\geq(\gamma+\eta)t+L,\ x\in\Omega,
\end{equation}
or

\item the nonlinearity $f$ is of the monostable type \eqref{monostable} and we have:

\begin{equation}\label{compmono}
\exists L\in\mathbb{R},\ \underline{u}(t,x)\leq0\ \text{ if } \ t\leq0,\ x\cdot e\geq(\gamma+\eta)t+L,\ x\in\Omega.
\end{equation}
\end{itemize}

Then, the following comparison result holds
\[
\underline{u}(t,x)\leq\overline{u}(t,x),\ \ \forall t\in \mathbb{R},\ \forall x\in\Omega.
\]

\end{lemma}
In addition to those two technical lemmas, we shall need the following result, stating that, in our framework, the speed of invasion $w$ is a continuous function:

\begin{lemma}\label{prop1}
Let $A , q ,f ,\Omega$ be periodic, satisfying  \eqref{regularity}-\eqref{hypf}. Assume that $f$ is of the monostable type \eqref{monostable} or of the combustion type \eqref{combustion}. Let $w$ be defined by \eqref{FGformula}. Then $w$ is a continuous function from the sphere $\mathbb{S}^{N-1}$ to $\mathbb{R}_{+}$. 
\end{lemma}

It is proven in  \cite[Proposition 2.6]{R1} that, if $c : \mathbb{S}^{N-1} \to \mathbb{R}_{+}$ is such that $\inf c >0$ and if $w$ is defined by $w(\xi) := \inf_{e\cdot \xi >0}\frac{c(e)}{e \cdot\xi}$, then $w$ is continuous. Hence, Lemma \ref{prop1} comes directly if we can prove that $\inf c^{\star}>0$, where $c^{\star}$ is the critical speed of fronts. To prove this, it is sufficient to show that $c^{\star}$ is lower semicontinuous. This is done in \cite[Proposition 2.5]{R1} when $\Omega = \mathbb{R}^{N}$, and the proof can be readily adapted, hence we will not prove Lemma \ref{prop1} but refer the reader to \cite{R1}. Let us mention another result of independent interest by Alfaro and Giletti \cite{AG} in the case where $\Omega = \mathbb{R}^{N}$, which states that, under suitable assumptions, $c^{\star}$ is actually continuous.

\subsection{Proof of Theorem \ref{FG}}

This section is dedicated to the proof Theorem \ref{FG}. We show that the Freidlin-Gartner formula \eqref{FGformula} still holds in the context of periodic domains $\Omega$ considered in this paper. The proof is divided in several steps. The main idea is to use a geometric argument, introduced in \cite{R1} : from an initial datum that invades space, we construct a front-like solution of our problem, and we compare it to pulsating traveling fronts.

\begin{proof}

We start to prove that $\mathcal{W}$, defined by \eqref{W}, is an asymptotic subset of spreading. We argue by contradiction. We assume that $\mathcal{W}$ is not an asymptotic subset of spreading. Then, there is a compact set $K \subset \text{int}(W)$ such that \eqref{subset} does not hold. Now, we take $W\subset \mathcal{W}$, $W$ star-shaped with respect to the origin, compact and $C^{\infty}$ such that $K \subset \text{int}(W)$. We assume that $W$ is the graph of a function $\tilde{w}$, i.e., $W=\left\{ r\xi \ : \  \xi \in \mathbb{S}^{N-1}, \ 0\leq r \leq \tilde{w}(\xi) \right\}$, with $\tilde{w}$ smooth and $\tilde{w}<w$, so that $W$ is strictly contained in $\mathcal{W}$. We take $\tilde{w}$ strictly positive. This is possible because the function $w$ is continuous thanks to Lemma \ref{prop1}.

The set $W$ satisfies the uniform interior ball estimates : $\exists \rho >0$ such that $\forall x \in \partial W, \, \exists y \in W$ such that $\overline{B}_{\rho}(y) \subset W$ and $x \in \partial B_{\rho}(y)$, where $B_{\rho}(y)$ is the ball of center $y$ and of radius $\rho$. In the course of the proof, $u(t,x)$ denotes a solution of \eqref{eqgen} arising from a non-negative, compactly supported initial datum such that invasion occurs, i.e., $u(t,x) \to 1$ as $t$ goes to $+\infty$, locally uniform in $x\in \overline{\Omega}$.

\emph{First step. Definition of $\mathcal{R}^{\eta}$.} \\
Let $0<\eta<1$. We define
$$
\mathcal{R}^{\eta}(t)\ :=\sup\{r\geq0 \ : \ \forall x\in(rW)\cap\overline{\Omega},\ u(t,x)>\eta\}.
$$
For $t\geq 0$, this quantity is well defined because $u(t,x)$ decays to zero as $\vert x \vert$ goes to $+\infty$ (this comes easily by comparison with pulsating traveling fronts) implying that $\mathcal{R}^{\eta}(t) < +\infty$. Moreover, we have that $\mathcal{R}^{\eta}(t)\to+\infty$ as $t$ goes to $+\infty$ (because of the assumption that $u(t,x) \to 1$  locally uniformly in $x$ when $t \to +\infty$).

Remembering that we assumed, by contradiction, that there is a compact set $K\subset \text{int}(W)$ such that \eqref{subset} does not hold, we can infer that there are $\eta,k\in(0,1)$ such that
\begin{equation}\label{R}
\liminf_{t\to+\infty}\frac{\mathcal{R}^{\eta}(t)}{t}<k.
\end{equation}
Indeed, if this were not the case, then $\forall\eta\in(0,1),\ \liminf_{t \to +\infty} \frac{\mathcal{R}^{\eta}(t)}{t}\geq1$. Hence, taking $h \in (0,1)$ such that $K \subset hW$, we have
\begin{equation*}
\begin{array}{llc}
\eta &\leq \displaystyle\liminf_{t \to +\infty} \inf_{x \in R^{\eta}(t)W}u(t,x) \\
&\leq \displaystyle\displaystyle\liminf_{t \to +\infty} \inf_{x \in htW}u(t,x) \\
&\leq \displaystyle\displaystyle\liminf_{t \to +\infty} \inf_{x\in tK}u(t,x).
\end{array}
\end{equation*}
This being true for each $\eta \in (0,1)$, it would yield that $K$ satisfies \eqref{subset}, which we assumed not to be the case. Observe that \eqref{R} is still verified if we increase $\eta$. We do so, and in the following we assume that $\eta \in (S,1)$, where $S$ is defined in \eqref{hypf}.

 From now on, we simplify our notations by writing $\mathcal{R}$ instead of $\mathcal{R}^{\eta}$. Observe that $\mathcal{R}$ is lower semicontinuous. Indeed, let $t_{n}$ be a sequence such that $t_{n}\to t_{0}$ as $n$ goes to $+\infty$ and such that $\mathcal{R}(t_{n})\to R\in \mathbb{R}$. Consider $r>R$. Then, for $n$ large enough, we have that $r>\mathcal{R}(t_{n})$, and, by definition of $\mathcal{R}(t_{n})$, there is $x_{n}\in (rW)\cap \overline{\Omega}$ such that $u(t_{n},x_{n})\leq\eta$. By continuity of $u$, there is some $x_{0}\in (rW)\cap\overline{\Omega}$ such that $u(t_{0},x_{0})\leq\eta$. This implies that $\mathcal{R}(t_{0})\leq r$ , and then that $\mathcal{R}(t_{0})\leq R$ by arbitrariness of $r>R$, hence the semicontinuity.

\emph{Second step. Shifting the function.} \\
By definition of $\mathcal{R}$ we have that $\liminf_{t\to+\infty}(\mathcal{R}(t)-kt)=-\infty$. We define, for $n\in \mathbb{N}$,
$$
 t_{n}\,:=\inf\{t\geq0\ :\ \mathcal{R}(t)-kt\leq-n\}. 
$$
The lower semicontinuity of $\mathcal{R}$ (proven in the first step) gives us that the above infimum is a minimum, i.e., that $\mathcal{R}(t_{n})-kt_{n}\leq-n<\mathcal{R}(t)-kt$, $\forall t<t_{n}$, and that $t_{n}\to+\infty$ as $n$ goes to $+\infty$. Hence, the sequence $(t_{n})_{n \in \mathbb{N}}$ satisfies:

$$
\lim_{n\to+\infty}t_{n}=+\infty \ \text{ and } \ \forall n\in\mathbb{N},\ \forall t\in[0,t_{n}), \quad \mathcal{R}(t_{n})-k(t_{n}-t)<\mathcal{R}(t).
$$
%%%
Now, by definition of $\mathcal{R}(t)$, we have that $ \forall r>\mathcal{R}(t),\ \exists \: x_{r} \in \left( rW\cap \overline{\Omega}\right)\backslash\left((\mathcal{R}(t)W) \cap \overline{\Omega}\right)$ such that $u(t,x_{r})\leq \eta $. Up to extraction, we can assume that $x_{r} \to x_{\infty}$ as $r$ goes to $\mathcal{R}(t)$, where $x_{\infty}\in  \overline{\Omega}\cap \partial\left( \mathcal{R}(t)W\right)$. 
By continuity, we have that $u(t,x_{\infty})=\eta$.

Hence, we can consider a sequence $(x_{n})_{n\in \mathbb{N}} \in \overline{\Omega}$ such that $u(t_{n},x_{n})=\eta$, with the additional property that $x_{n}\in \partial \left(\mathcal{R}(t_{n})W\right)$. Clearly, $\vert x_{n}\vert \to+\infty$ as $n$ goes to $+\infty$. If $x\in \partial W$, let $\tilde{\nu}(x)$ the outer unit normal to $W$ at the point $x$. We define
$$
\hat{x}_{n}=\frac{x_{n}}{\mathcal{R}(t_{n})} , \quad y_{n}= \hat{x}_{n}-\rho \tilde{\nu} (\hat{x}_{n}).
$$
By definition, $\hat{x}_{n} \in \partial W$ and $y_{n}$ is the center of the interior ball tangent at $W$ at point $\hat{x}_{n}$, of radius $\rho$ (we recall that $W$ satisfies the uniform interior ball estimate with radius $\rho$).

For every $n$, we define $k_{n} \in \mathbb{Z}^{N} $ and $z_{n} \in [0,1)^{N}$ by $ x_{n}=k_{n}+z_{n}$. Up to extraction, we can assume that there is $z \in [0,1]^{N}$ such that $z_{n} \to z$ as $n\to+\infty$. We also assume that there is $\hat{x}$ such that $\hat{x}_{n}$ converges to $\hat{x}$, whence $\tilde{\nu} (\hat{x}_{n})$ converges to $\tilde{\nu} (\hat{x})$. We now define, for $n\in \mathbb{N}$, the translated functions: 
$$
u_{n}(t,x)=u(t+t_{n},x+k_{n}).
$$
Thanks to the periodicity and regularity hypotheses on $\Omega$, we can apply the usual interior and portion boundary parabolic estimates (see, for instance \cite[Theorems 5.2, 5.3]{La}) to get that $u_{n}$ converges uniformly locally to an entire solution $u^{\star} $ of the equation \eqref{eqgen}. Moreover $u^{\star}(0,z)= \eta$.

\emph{Third step. Properties of $u^{\star}$.}\\
 We show here that $u^{\star}$ is a front-like solution, in the sense that it satisfies, denoting $ H_{T} := \{ x\in \Omega \ :  \  x \cdot \tilde{\nu}(\hat{x}) < -k \hat{x} \cdot \tilde{\nu} ( \hat{x} ) T \} $:
\begin{equation}\label{inv}
\forall T\geq0,\ \forall x\in H_{T}+\left\{z\right\},\quad  u^{\star}(-T,x)\geq\eta.
\end{equation}
To show this, take $T\in[0,t_{n}]$ and $x\in(\mathcal{R}(t_{n})-kT)W\cap\Omega$. As $\mathcal{R}(t_{n})-kT\leq\mathcal{R}(t_{n}-T)$, we have that
$x\in\mathcal{R}(t_{n}-T)W\cap\Omega$. Therefore, by definition of $\mathcal{R}$, $u(t_{n}-T,x)\geq\eta$.
Then, we have

$$
\forall T\in[0,t_{n}], \ \forall x\in\left((\mathcal{R}(t_{n})-kT)W\right)\cap\Omega-\left\{ k_{n} \right\} , \quad u_{n}(-T,x)\geq\eta.
$$
From that, we infer:
$$
\forall T\geq 0 ,\ \forall x\in \Omega\cap\bigcup_{M \in \mathbb{N}} \bigcap_{n \geq M}( (\mathcal{R}(t_{n})-kT)W-\{k_{n}\}) , \quad u^{\star}(-T,x)\geq\eta.
$$
To prove \eqref{inv}, it suffices to show that $H_{T}+\{z\} \subset \Omega\cap\bigcup_{M \in \mathbb{N}} \bigcap_{n \geq M}( (\mathcal{R}(t_{n})-kT)W-\{k_{n}\})$.
To see this, take $ x \in H_{T}+\{z\}$. We start to compute:

 \begin{align*}
\left| \frac{x+k_{n}}{\mathcal{R}(t_{n})-kT} - y_{n} \right| &= \left| \frac{x+k_{n}-(\mathcal{R}(t_{n})-kT)(\hat{x}_{n}-\rho \tilde{\nu} (\hat{x}_{n}))}{\mathcal{R}(t_{n})-kT} \right| \\ 
&= \left| \frac{x+kT\hat{x}_{n}+ (k_{n} -x_{n})+(\mathcal{R}(t_{n})-kT)\rho \tilde{\nu} (\hat{x}_{n})}{\mathcal{R}(t_{n})-kT} \right| \\
&= \left| \rho \tilde{\nu} (\hat{x}_{n}) + \frac{x+kT\hat{x}_{n}-z_{n}}{\mathcal{R}(t_{n})-kT} \right|  .
\end{align*}
Let us call $ w_{n}:=\frac{x+kT\hat{x}_{n}-z_{n}}{\mathcal{R}(t_{n})-kT}$.
This goes to zero as $n$ goes to infinity. The last term in the above equality can be rewritten
$\vert \rho \tilde{\nu} (\hat{x}_{n})+w_{n} \vert= \sqrt{\rho^{2}+\vert w_{n} \vert (2\rho \tilde{\nu} (\hat{x}_{n})\cdot w_{n}/\vert w_{n} \vert+ \vert w_{n} \vert)}$. Now, observe that
$$
\lim_{n\to +\infty} 2\rho \tilde{\nu} (\hat{x}_{n})\cdot w_{n}/\vert w_{n} \vert+ \vert w_{n} \vert = 2 \rho \tilde{\nu} (\hat{x}) \cdot \frac{x+kT\hat{x}-z}{\vert x+kT\hat{x}-z \vert}.
$$
This limit is strictly negative. Indeed, if $x \in H_{T}+\{z\} $, then $(x-z)~\cdot~\tilde{\nu}(\hat{x})~<~-kT \hat{x}\cdot~\tilde{\nu}(\hat{x}) $.
Therefore, we have, for $n$ large enough, 
$$
 \left| \frac{x+k_{n}}{\mathcal{R}(t_{n})-kT} - y_{n} \right|< \rho, 
 $$
 which means that $\frac{x+k_{n}}{\mathcal{R}(t_{n})-kT} \in W $, by definition of $y_{n}$ and $\rho$. In other words, $ x~\in~(\mathcal{R}(t_{n})-kT)W-\{k_{n}\} $, which concludes this step.

\emph{Fourth step. Comparison.} \\
We now compare the function $u^{\star}$ built in the previous steps to the pulsating front traveling in the direction $\tilde{\nu}(\hat{x})$ with critical speed $c^{\star}(\tilde{\nu}(\hat{x}))$. Combining Lemma \ref{lemma1} and \eqref{inv}, we have that

$$
\liminf_{\delta\to+\infty}\inf_{\substack{t<0  \\  x\cdot \tilde{\nu}(\hat{x})< \gamma t -\delta  \\ x\in\Omega }} u^{\star}(t,x) \geq 1 ,
$$
with $\gamma := k \hat{x}\cdot \tilde{\nu}(\hat{x})>0$. Hence $u^{\star}$ satisfies the hypotheses of Lemma \ref{lemma2}. Observe that we have

\begin{equation*}
\begin{array}{lll}
\gamma &= k \hat{x}\cdot \tilde{\nu}(\hat{x}) \\ &=k \frac{\hat{x}}{\vert \hat{x} \vert }\cdot \tilde{\nu}(\hat{x})\tilde{w}\left(\frac{\hat{x}}{\vert \hat{x} \vert }\right) \\
 &< \frac{\hat{x}}{\vert \hat{x} \vert }\cdot \tilde{\nu}(\hat{x})w\left(\frac{\hat{x}}{\vert \hat{x} \vert }\right) \\ &\leq c^{\star}(\tilde{\nu}(\hat{x})), 
\end{array}
\end{equation*}
where the last inequality follows from the definition of $w$ in Theorem \ref{FG}.

Assume first that $f$ is of the combustion type \eqref{combustion}. Let $v$ be a pulsating traveling front in the direction $\nu(\hat{x})$, with critical speed $c^{\star}(\nu(\hat{x}))$. Up to a time translation, we normalize it so that $v(0,0)>u^{\star}(0,0)$. Then, $v$ satisfies the hypotheses of Lemma \ref{lemma2} (with $\eta = c^{\star}(\nu(\hat{x}))-\gamma $ in the hypotheses of Lemma \ref{lemma2}), giving $v\leq u^{\star}$, which is in contradiction with the fact that $v(0,0)>u^{\star}(0,0)$. Hence the contradiction.

Now, if the nonlinearity is of the monostable type \eqref{monostable}, we have to build a function $v$ satisfying \eqref{compmono} to apply Lemma \ref{lemma2}. This can be done exactly as in \cite[Proposition 2.6]{R1}, the fact that the domain is not $\mathbb{R}^{N}$ adds no difficulty here. 
This proves that $W$ is an asymptotic subset of spreading, and then, so is $\mathcal{W}$. Now, we show that it is an asymptotic superset of spreading.

\emph{Fifth step. Superset of spreading.} \\
Let $C$ be a closed set such that $\mathcal{W}\cap C = \emptyset $. Then, because $w$ is continuous, we can find $\varepsilon>0$ so that $\mathcal{W}_{\varepsilon} := \left\{ r\xi \ , \ \xi \in \mathbb{S}^{N-1},\ 0 \leq r \leq w(\xi) +\varepsilon\right\}$ is such that $\mathcal{W}_{\varepsilon} \cap C = \emptyset$. To prove that  $\mathcal{W}$ is an asymptotic superset of spreading, it is sufficient to show that $\sup_{x \in t\mathcal{W}_{\varepsilon}^{c}} u(t,x) \to 0$ as $t$ goes to $+\infty$. To do so, we take a sequence $(t_{n})_{n\in\mathbb{N}} \in (\mathbb{R}^{+})^{\mathbb{N}}$ such that $t_{n}$ goes to infinity as $n$ goes to infinity, and a sequence $x_{n} \in t_{n}\mathcal{W}_{\varepsilon}^{c}$ such that 
$$
u(t_{n},x_{n}) \geq \frac{1}{2}\sup_{x \in t_{n}\mathcal{W}_{\varepsilon}^{c}} u(t_{n},x).
$$
Up to extraction, we take $e\in \mathbb{S}^{N-1}$ such that $\frac{x_{n}}{\vert x_{n} \vert} \to e$ as $n$ goes to $+\infty$. Let $\xi \in \mathbb{S}^{N-1}$ be such that $w(e) = \frac{c^{\star}(\xi)}{\xi\cdot e}$, and let $v$ be a pulsating traveling front in the direction $\xi$ with critical speed $c^{\star}(\xi)$. Up to some translation in time, we can assume, thanks to the parabolic comparison principle, that $u(t,x)\leq v(t,x)$, $\forall t\geq0,\ \forall x \in \Omega$. Let us show that $v(t_{n},x_{n}) $ goes to zero as $n\to +\infty$.

We write $x_{n} := \left( \frac{x_{n}}{\vert x_{n} \vert} \cdot \xi\right)\vert x_{n} \vert \xi +d_{n}$, where $d_{n}$ is orthogonal to $\xi$. Because $\frac{x_{n}}{\vert x_{n} \vert} \to e$ as $n$ goes to $+\infty$, using the continuity of $w$, for $n$ large enough, we have 
$$
\left( \frac{x_{n}}{\vert x_{n} \vert} \cdot \xi\right)\vert x_{n} \vert \geq \left( \frac{x_{n}}{\vert x_{n} \vert} \cdot \xi\right)(w(\frac{x_{n}}{\vert x_{n} \vert}) +\varepsilon )t_{n}  \geq (c^{\star}(\xi) +(e\cdot \xi)\frac{\varepsilon}{2})t_{n}.
$$
 So, we get that, for $n\in \mathbb{N}$ large enough, there is some $\lambda_{n}$, such that $\lambda_{n} \geq c^{\star}(\xi) + (e\cdot \xi)\frac{\varepsilon}{2}$ and $x_{n} = \lambda_{n}\xi t_{n} + d_{n}$. Now, observe that the definition of the pulsating traveling fronts, Definition \ref{defPTF}, implies that $v(t_{n}, \lambda_{n}t_{n}\xi + d_{n}) \to 0$ as $n$ goes to $+\infty$, hence 
 $$
 \lim_{n\to+\infty}\frac{1}{2}\sup_{x \in t_{n}\mathcal{W}_{\varepsilon}^{c}} u(t_{n},x)\leq \lim_{n\to +\infty}u(t_{n},x_{n}) \leq \lim_{n\to +\infty}v(t_{n},x_{n}) = 0,
 $$
 which implies the result.

\end{proof}

Now that we dispose of the Freidlin-Gartner formula \eqref{FGformula}, we use it to answer Question \ref{question}.

\section{Estimates for the spreading speed}\label{speeds}

This whole section is dedicated to the proof of Theorem \ref{thwc}. We consider here the problem \eqref{PeriodicFKPP}, with nonlinearity $f$ independent of $x$ of the monostable or combustion type. In the following, for $f$ and $\Omega$ given, we denote  $c^{\star}$ and $w$ the critical speed of fronts and the speed of invasion respectively, for equation \eqref{PeriodicFKPP}.

The proof of Theorem \ref{thwc} is done in several steps : first, we show that $w\equiv c^{\star}$ is equivalent to saying that $w$ and $c^{\star}$ are actually constant. This is the object of Section \ref{section prop}. Then, we exhibit in Section \ref{geodesic estimates} some estimates on the spreading speed that take into account the geometry of the domain. Gathering all this, we will be able to prove Theorem \ref{thwc}.

\subsection{Comparison between $w$ and $c^{\star}$}\label{section prop}

This section is dedicated to proving that, if the critical speed of fronts $c^{\star}$ and the speed of invasion $w$ are everywhere equal, then they are constant. This uses only the Freidlin-Gartner formula \eqref{FGformula} proved in Section \ref{sectionFG}.

\begin{prop}\label{prop}
Assume that $\Omega$ is a smooth periodic domain satisfying \eqref{regomega} and that $f$ is a nonlinearity satisfying \eqref{hypf} of the monostable \eqref{monostable} or combustion \eqref{combustion} type. Assume that $c^{\star}\equiv w$. Then, the functions $w$ and $c^{\star}$ are constant.
\end{prop}

\begin{proof}
Because of the hypotheses on $\Omega$ and $f$, we can apply Theorem \ref{FG} to get that $\forall e \in \mathbb{S}^{N-1}$, $w(e) = \inf_{\xi \cdot e>0} \frac{c^{\star}(\xi)}{\xi \cdot e}$. Assume that $w \equiv c^{\star}$ and take $\xi_{0},\,\xi \in\mathbb{S}^{N-1}$ so that $\xi_{0}\cdot\xi>0$, and let $\omega$ be the angle between those two vectors. Let us take $M\in\mathbb{N}$. We define a sequence $(\xi_{k})_{k\in \llbracket 0 , M \rrbracket}\in \mathbb{S}^{N-1}$ to be equidistributed on the arc joining $\xi_{0}$ to $\xi$ on the sphere, i.e.,  $\xi_{k}\cdot\xi_{k+1}=\cos(\frac{\omega}{M})$ and $\xi_{M} = \xi$. Then, we have 
\[
w(\xi_{0})\leq w(\xi_{1})\frac{1}{\xi_{0}\xi_{1}}\leq w(\xi_{2})\frac{1}{\xi_{2}\xi_{1}}\frac{1}{\xi_{1}\xi_{0}}.
\]
Iterating and using that $\xi_{k}\cdot\xi_{k+1}=\cos(\frac{\omega}{M})$ , we get:
\[
w(\xi_{0})\leq w(\xi)\prod_{k = 0}^{M-1}\frac{1}{\xi_{k+1}\xi_{k}} =  w(\xi)\frac{1}{\cos(\frac{\omega}{M})^{M}}.
\]
Because $\frac{1}{\cos(\frac{\omega}{M})^{M}}\sim 1+ \frac{\omega^{2}}{2M}$ when $M$ goes to $ +\infty$, passing to the limit yields:
\[
w(\xi_{0}) \leq w(\xi).
\]
Inverting the roles of $\xi_{0}$ and $\xi$, we get $w(\xi_{0}) = w(\xi)$. Hence, $w$ is constant, and so is $c^{\star}$.

\end{proof}

Observe that, in the course of the proof, we did not use the particular form of equation \eqref{PeriodicFKPP}, only the Freidlin-Gartner formula, hence Proposition \eqref{prop} holds true also for the general equation \eqref{eqgen}.

As mentioned in the introduction, we shall use this result to build domains where $c^{\star} \not\equiv w$. Indeed, Proposition \ref{prop} reduces the problem to finding domains where $w$ or $c^{\star}$ are not constants. Intuitively, it seems that, if in a certain direction $e$, there are lot of ``obstacles", then the speeds $w$ and $c^{\star}$ should be ``small". On the contrary, if on a certain direction, there are few obstacles, then the speeds should be ``large". Hence, if the domain $\Omega$ is very ``obstructed" in some direction and not in an other, then the speeds should not be constants, and so they would be different.

To build such domains is actually quite easy if $f$ is KPP and if the dimension is greater or equal to $3$. In this case, we will see in the next Section \ref{domain invariant} that we can use domains invariants in one direction. If the nonlinearity is not KPP or if the dimension is equal to $2$, things are more involved. To overcome this difficulty, we introduce estimates for $w$ that do take into account the ``obstructions" of the domain. This is done in Section \ref{geodesic estimates}.

\subsection{Invasion in domains that are invariant in one direction}\label{domain invariant}

In this subsection, $f$ is a KPP nonlinearity independent of $x$ and $\Omega$ is invariant in the direction $e \in \mathbb{S}^{N-1}$, i.e., for all $\lambda \in \mathbb{R}$, we have $\Omega + \lambda e =\Omega$. Let us answer Question \ref{question} in this specific case by proving the following:

\begin{prop}\label{w=c}

Let $\Omega$ be a periodic domain in $\mathbb{R}^{N}$, $N\geq 3$, satisfying \eqref{regomega} and invariant in the direction $e \in \mathbb{S}^{N-1}$. Let $f$ satisfying \eqref{hypf} be a KPP nonlinearity independent of $x$. Denoting $c^{\star}$ and $w$ the critical speed of fronts and the speed of invasion respectively for problem \eqref{PeriodicFKPP}, we have

$$
w\equiv c^{\star}  \iff \Omega = \mathbb{R}^{N}.
$$
\end{prop}

This comes directly by combining our Proposition \ref{prop} with the following result from \cite{BHN1}:

\begin{theorem}
Let $c^{\star}$ be the critical speed of fronts for the problem \eqref{PeriodicFKPP} with $f$ KPP independent of $x$. Then $c^{\star}(e)\leq2\sqrt{f^{\prime}(0)}$ and the equality holds
if and only if $\Omega$ is invariant in the direction $e$.
\end{theorem}

If $\Omega$ is a periodic domain satisfying hypothesis \eqref{regomega} and invariant in a direction, $\Omega \neq \mathbb{R}^{N}$, then this Theorem implies that $c^{\star}$ is not constant (as function of the direction). Then, Proposition \ref{prop} implies that $c^{\star}\not\equiv w$. This answers Question \ref{question} in the special case where $f$ is KPP and the dimension greater than $3$. The general setting is more involved and is addressed after.

However, when considering domains invariant in one direction, we can give further informations about the shape of the asymptotic set of spreading $\mathcal{W}$. The next result shows that, if $\Omega$ is invariant in the direction $e$, then the spreading speed in a direction orthogonal to $e$ does only depend on the part of the domain orthogonal to $e$. More precisely, we have

\begin{prop}
Let $\Omega$ be a periodic domain satisfying \eqref{regomega}, invariant in the direction $e \in \mathbb{S}^{N-1}$. Let $\mathcal{W}$ be the asymptotic set of spreading of equation \eqref{PeriodicFKPP} set on $\Omega$ with $f$ satisfying \eqref{hypf} and such that $u \mapsto \frac{f(u)}{u}$ is decreasing (this implies that $f$ is KPP). Let $\mathcal{H}$ be the hyperplane in $\mathbb{R}^{N}$ orthogonal to $e$. Then, if $\mathcal{W}_{\mathcal{H}\cap \Omega}$ is the asymptotic set of spreading for the same equation restricted to $\mathcal{H}\cap \Omega$, i.e., 

\begin{equation}\label{equation restricted}
\left\{
\begin{array}{rrll}
\partial_{t}u - \Delta u &=&  f(u), &\quad t>0 , \ x\in \mathcal{H}\cap\Omega, \\
\partial_{\nu^{\prime}} u &=& 0,  &\quad t>0 ,\ x\in \partial (\mathcal{H}\cap\Omega),
\end{array}
\right.
\end{equation}
where $\nu^{\prime} \in \mathbb{S}^{N-2}$ denotes the exterior normal to $\mathcal{H}\cap\Omega$,
 we have
$$
\mathcal{W}_{\mathcal{H}\cap\Omega} = \mathcal{W}\cap\mathcal{H}.
$$
\end{prop}

\begin{proof}
To simplify the notations, we denote $w_{N}$ the spreading speed for the Fisher-KPP equation \eqref{PeriodicFKPP} set on $\Omega$ and  $w_{N-1}$ the spreading speed for the equation \eqref{equation restricted} set on $\mathcal{H}\cap \Omega$. Similarly, we denote $c^{\star}_{N}$ and $c^{\star}_{N-1}$ the critical speeds of fronts for the equation \eqref{PeriodicFKPP} and \eqref{equation restricted} respectively. Up to some rotation of the coordinates, we write the points of $\Omega$ under the form $(x,y)$, where $x\in\mathcal{H}\cap\Omega$ and $y \in \mathbb{R}$.

\emph{Step 1.} \\We start to show that, for each $\zeta \in \mathbb{S}^{N-2}$, we have $w_{N-1}(\zeta) \geq w_{N}((\zeta,0))$. To do so, take $\xi \in \mathbb{S}^{N-2}$ such that $\xi \cdot \zeta >0$.  Let $\phi_{\xi}(t,x)$ be a pulsating traveling front solution of \eqref{equation restricted} in the direction $\xi $ with critical speed $c^{\star}_{N-1}(\xi)$. For $(x,y)\in \Omega$, we define $\Phi(t,x,y) := \phi_{\xi}(t,x)$. Then $\Phi$ is solution of the equation \eqref{PeriodicFKPP} on the whole of $\Omega$. If $u_{0}(x,y)$ is a non-negative compactly supported initial datum and if $u(t,x,y)$ is the solution of \eqref{PeriodicFKPP} arising from it, we can assume that (up to translation) $u_{0}(x,y) \leq \Phi(0,x,y)$. Hence, the parabolic comparison principle gives us that 
$$
u(t,x,y)\leq \Phi(t,x,y), \quad \forall t \geq 0, \ \forall (x,y) \in \Omega.
$$
 Observe that $\Phi$ moves in the direction $(\zeta , 0) \in \mathbb{S}^{N-1}$ with speed $\frac{c^{\star}_{N-1}(\xi)}{\xi \cdot \zeta}$. This means that $w_{N}((\zeta,0)) \leq \frac{c^{\star}_{N-1}(\xi)}{ \xi \cdot \zeta}$, and because this is true for all $\xi$ such that $\xi \cdot \zeta >0$, Theorem \ref{FG} implies that $w_{N}((\zeta,0))\leq w_{N-1}(\zeta)$.

\emph{Step 2.}\\ We now prove the reverse inequality. To start, let $\varepsilon>0$ be fixed such that $\varepsilon^{2} < f^{\prime}(0)$. We define a KPP nonlinearity $f_{\varepsilon}(u) := f(u) -\varepsilon^{2}u$. Let $u_{0}(x)$ be a smooth, non-negative, compactly supported function in $\mathcal{H}\cap\Omega$. Let $u_{\varepsilon}(t,x)$ be the solution arising from $u_{0}$  of \eqref{PeriodicFKPP} but \emph{with $f$ replaced by $f_{\varepsilon}$}.

Define the cut-off function
\begin{equation*}
\phi(y) :=  \left\{ \begin{array}{lll}   \cos (\varepsilon y) &\text{ for }& \vert y \vert \leq \frac{\pi}{2\varepsilon}  \\      0  &\text{ for }&  \vert y \vert \geq \frac{\pi}{2\varepsilon}.\end{array}\right.
\end{equation*}

Now, let $v(t,x,y) := u_{\varepsilon}(t,x)\phi( y)$. Let us show that $v$ is a (generalized) subsolution. An easy computation shows that, for $(x,y)\in \Omega$ such that $v(t,x,y) >0$, we have
\begin{equation*}
\begin{array}{lll}
\partial_{t}v - \Delta v -f(v) &=& f_{\varepsilon}(u_{\varepsilon})\phi(y)+\varepsilon^{2}u_{\varepsilon}\phi(y)-f(u_{\varepsilon}\phi) \\
 &=& \left(\frac{f_{\varepsilon}(u_{\varepsilon})}{u_{\varepsilon}} - \frac{f(u_{\varepsilon}\phi)}{u_{\varepsilon}\phi}+\varepsilon^{2}\right)u_{\varepsilon}\phi \\
 &\leq& 0.
\end{array}
\end{equation*}
The last inequality comes from the fact that $z \mapsto \frac{f(z)}{z}$ is decreasing. One can then check that $\partial_{\nu}v = 0$ on $\partial \Omega$. This comes from $\partial_{\nu^{\prime}}u_{\varepsilon} = 0$ on $\partial (\Omega\cap\mathcal{H})$ together with the fact that $\Omega$ is invariant in the direction $e$.

Hence, $u_{\varepsilon}\phi$ is a (generalized) subsolution of \eqref{PeriodicFKPP} (with nonlinearity $f$). We can observe that $u_{\varepsilon}$ spreads in $\Omega\cap \mathcal{H}$ in the direction $\zeta \in \mathbb{S}^{N-2}$ with speed $w_{N-1}(\zeta)-\varepsilon^{2}$.
Hence, by comparison, we get that $w_{N-1}(\zeta)-\varepsilon^{2} \leq w_{N}((\zeta,0))$. Taking the limit $\varepsilon \to 0$ yields the result.
\end{proof}

Now, we turn to the full proof of Theorem \ref{thwc}, answering then Question \ref{question}.

\subsection{Geodesic estimates}\label{geodesic estimates}

This aim of this section is to establish estimates on $w(e)$ that do take into account the geometry of the domain. The key tool is an estimate on the heat kernel from \cite{BHN2}, following from general results on the heat kernel by Davies \cite{Davies} and Grigor'yan \cite{G}. This estimate is valid for domains satisfying the \emph{extension property}. Denoting $W^{1,p}(\Omega)$ the usual Sobolev space over $\Omega$, a non-empty subset of $\mathbb{R}^{N}$ satisfies the extension property if, for all $1\leq p \leq + \infty $, there is a bounded linear map $E$ from $W^{1,p}(\Omega)$ to $W^{1,p}(\mathbb{R}^{N})$ such that $E(f)$ is an extension of $f$ from $\Omega$ to $\mathbb{R}^{N}$, for all $f\in W^{1,p}(\Omega)$. For our purpose, we mention that the smooth periodic domains we consider here satisfy the extension property, see \cite{Stein}.

\begin{prop}
Let $\Omega$ be a locally $C^{2}$ non-empty connected open subset of $\mathbb{R}^{N}$ satisfying the extension property. Let $p(t,x,y)$ be the heat kernel in $\overline{\Omega}$ with Neumann boundary condition on $\partial \Omega$. Then, for every $\varepsilon>0$, there are two positive constants $C$ and $\delta$ such that 
 \begin{equation}\label{heat}
 \forall t>0, \ \forall (z,x) \in \overline{\Omega}\times \overline{\Omega}, \quad p(t,z,x) \leq C(1+\delta t^{-\frac{N}{2}}) \exp \left( -\frac{d_{\Omega}(z,x)^{2}}{(4+\varepsilon)t} \right) ,
 \end{equation}
 where $d_{\Omega}(z,x)$ denotes the geodesic distance in $\overline{\Omega}$.
\end{prop}
See \cite[Proposition 2.5]{BHN2} for the proof. We use this to get upper estimates on the spreading speed $w(e)$. To do so, we introduce the following coefficient, for $e \in \mathbb{S}^{N-1}$ :
\begin{equation}\label{C geo}
C_{\Omega}(e) := \liminf_{\lambda \to + \infty} \frac{\lambda}{d_{\Omega}(0,\lambda e)}.
\end{equation}
For notational simplicity and without loss of generality, we assume that the point $0$ is in $\Omega$. Up to translation, this is always possible, and will be be assumed in the following.

This coefficient represents \emph{how much the domain is obstructed} in the direction $e$. The geodesic distance $d_{\Omega}$ is always greater than the 
euclidian distance, hence $C_{\Omega}(e) \leq 1$.

\begin{prop}\label{geo}
Let $\Omega$ be a domain satisfying \eqref{regomega} and $f$ a monostable \eqref{monostable} or a combustion \eqref{combustion} nonlinearity independent of $x$. 
We denote $w$ the speed of invasion associated to problem \eqref{PeriodicFKPP}. Then, we have
\begin{equation}\label{geodesic}
w(e) \leq 2C_{\Omega}(e)\sqrt{ \max_{u\in[0,1]}\frac{f(u)}{u} }.
\end{equation}
\end{prop}
Observe that, if $f$ is a KPP nonlinearity, then this formula boils down to $w(e) \leq 2C_{\Omega}(e)\sqrt{f^{\prime}(0)}$. In the case where $\Omega = \mathbb{R}^{N}$, the upper bound is actually the KPP speed $2\sqrt{f^{\prime}(0)}$.

\begin{proof}

Let us observe that it is sufficient to prove the result in the KPP case. Indeed, if $f$ is a monostable or a combustion nonlinearity, then there is a KPP nonlinearity $\overline{f}$ such that $\overline{f}^{\prime}(0) = \max_{u\in [0,1]}\frac{f(u)}{u}$ and $\overline{f} \geq f$. If $u_{0}$ is an initial datum, denoting $u$, respectively $\overline{u}$, the solution of \eqref{PeriodicFKPP} with nonlinearity $f$, respectively $\overline{f}$, arising from $u_{0}$, the parabolic comparison principle tells us that 
$$
u(t,x) \leq \overline{u}(t,x), \quad \forall t \geq 0, \ \forall x\in \Omega.
$$
Then, $w(e) \leq \overline{w}(e), \ \forall e \in \mathbb{S}^{N-1}$, where $w$, respectively $\overline{w}$, is the invasion speed for \eqref{PeriodicFKPP} with nonlinearity $f$, respectively $\overline{f}$. Then, it is sufficient to prove the estimate \eqref{geodesic} for $\overline{f}$, because $\max_{u \in [0,1]} \frac{\overline{f}(u)}{u}=\max_{u \in [0,1]} \frac{f(u)}{u}$. Hence, in the rest of the proof, we assume that $f$ is KPP, and then $\max_{u \in [0,1]}\frac{f(u)}{u} = f^{\prime}(0)$.

Let $u(t,x)$ be the solution of the parabolic problem \eqref{PeriodicFKPP} arising from a compactly supported non-negative initial smooth datum $u_{0}$. Let $K$ be a compact set of $\Omega$ such that the support of $u_{0}$ is in $K$. We denote by $p(t,x,z)$ the heat kernel with Neumann condition on $\Omega$. Then, we first observe that 

\begin{equation}\label{estimate}
u(t,x) \leq e^{f^{\prime}(0)t}\int_{\Omega} p(t,x,z)u_{0}(z)dz.
\end{equation}
Indeed, $e^{f^{\prime}(0)t}\int_{\Omega} p(t,x,z)u_{0}(z)dz$ is the solution of the linearized problem

\begin{equation}\label{linearized}
\left\{
\begin{array}{rllr}
\partial_{t} v-\Delta v  &=f^{\prime}(0)v, &\quad t>0 , \ x \in  \Omega,    \\  
\partial_{\nu}v&=0, & \quad t>0 ,\ x \in  \partial  \Omega, \\
v(0,x) &= u_{0}(x), & \quad x \in \Omega,
\end{array}
\right.
\end{equation}
and hence is a supersolution of \eqref{PeriodicFKPP}, thanks to the KPP property. Then, the inequality \eqref{estimate} follows by the parabolic comparison principle. Now, let $\varepsilon>0$ be fixed. Using the estimate \eqref{heat} in \eqref{estimate}, we get

\begin{equation}\label{estimate2}
u(t,x) \leq C(1+\delta t^{-\frac{N}{2}}) e^{f^{\prime}(0)t}\int_{\Omega}\exp \left( -\frac{d_{\Omega}(z,x)^{2}}{(4+\varepsilon)t} \right)u_{0}(z)dz,
\end{equation}
for some positive constants $C$ and $\delta$ (depending on $\varepsilon$). This gives us

\begin{equation}\label{estimate3}
u(t,x) \leq C\| u_{0} \|_{L^{1}}(1+\delta t^{-\frac{N}{2}})  \exp \left( \left( f^{\prime}(0)-\frac{(\min_{z\in K}d_{\Omega}(z,x))^{2}}{(4+\varepsilon)t^{2} }\right)t \right).
\end{equation}
Now, take $e\in \mathbb{S}^{N-1}$ and $\omega>0$ such that $\omega<w(e)$. Then, $u(t,\omega te) \to 1$ as $t \to +\infty$, by definition of $w(e)$. Then, necessarily, we have
$$
\limsup_{t\to +\infty} \frac{\inf_{z\in K}d_{\Omega}(z,\omega te)}{t } \leq \sqrt{(4+\varepsilon)f^{\prime}(0)},
$$
if this were not the case, up to subsequence the right-hand term of \eqref{estimate3} would go to zero along some time sequence $(t_{n})_{n \in \mathbb{N}}$, $t_{n} \to +\infty$ as $n$ goes to $+\infty$, which would be in contradiction with the fact that $u(t_{n},\omega t_{n}e)$ goes to $1$. Using the triangular inequality for $d_{\Omega}$ and the fact that $K$ is compact we get
$$
\omega \leq \frac{\sqrt{(4+\varepsilon)f^{\prime}(0)}}{\limsup_{t\to +\infty} \frac{d_{\Omega}(0, \omega te)}{\omega t }}.
$$
Recalling the definition of $C_{\Omega}(e)$ and that the above inequality is true for every $\varepsilon>0$, we get
$$
\omega \leq 2C_{\Omega}(e)\sqrt{f^{\prime}(0)},
$$
and the result follows.
\end{proof}

We are now in position to answer Question \ref{question}.

\subsection{Domains where $c^{\star}\not\equiv w$}\label{c et w differ}

In this section, we build periodic domains $\Omega$ such that $c^{\star} \not\equiv w$. If $f$  is a KPP nonlinearity, we exhibit a $1-$periodic domain (but the periodicity can be chosen arbitrary). If $f$ is a monostable or a combustion nonlinearity, we build a $L$-periodic domain, where $L>0$ can be large. For clarity, we do this in dimension $N=2$, but these constructions can be easily generalized to greater dimensions.

In the following, we denote $e_{x} := (1,0), e_{y}:= (0,1) \in \mathbb{S}^{1}$ the unit vectors of the canonical basis of $\mathbb{R}^{2}$. Moreover, we define $e_{d} := \frac{1}{\sqrt{2}}(1,1) \in \mathbb{S}^{1}$.

\subsubsection{The KPP case}

We show here the following:
\begin{prop}\label{AnsKPP}
Let $f$ be a KPP nonlinearity \eqref{defkpp}. There is a smooth periodic domain $\Omega \subset \mathbb{R}^{2}$ such that $c^{\star}(e_{x}) > w(e_{d})$, where $c^{\star}$ and $w$ are the critical speed of fronts and the speed of invasion respectively for \eqref{PeriodicFKPP} set in $\Omega$ with nonlinearity $f$.
\end{prop}

We see that in this domain, it is not possible that $w = c^{\star}$, thanks to Proposition \ref{prop}. Hence, this answers Question \ref{question} in the KPP case.

\begin{proof}
For $\alpha \in (\frac{1}{2},1)$, $\beta \in (0,\frac{1}{2})$, we define $\Omega_{\alpha,\beta}$ to be a smooth periodic domain such that 
\begin{equation}\label{domain1}
\mathbb{Z}^{2}+(1-\alpha, \alpha)\times [\beta , 1-\beta] \subset \Omega_{\alpha,\beta}^{c} \subset \mathbb{Z}^{2}+\left( \frac{1-\alpha}{2} , \frac{1+\alpha}{2} \right)\times [\beta , 1-\beta].
\end{equation}
This domain is simply $\mathbb{R}^{2}$ with ``almost square" holes. For $\alpha, \beta$ given we denote $c^{\star}_{\alpha,\beta}(e)$ the critical speed of fronts in this domain in the direction $e$. If $\beta$ is fixed and if we let $\alpha \to 1$, then the domain ``converges" in some sense to an array of parallel disconnected strips in the direction $e_{x}$. This observation is made rigorous by \cite[Theorem 1.4]{BHN1}, where it is proven that:
$$
c^{\star}_{\alpha,\beta}(e_{x}) \underset{\alpha \to 1}{\longrightarrow} 2\sqrt{f^{\prime}(0)}.
$$
Now, let $\kappa \in (1,\sqrt{2})$ and take $\alpha$ close enough to $1$ so that $c^{\star}_{\alpha,\beta} (e_{x})> \frac{1}{\kappa}2\sqrt{f^{\prime}(0)}$. 

Take $n \in \mathbb{N}$. Denoting $d_{\Omega_{\alpha,\beta}}$ the geodesic distance in $\Omega_{\alpha,\beta}$, it is easy to see that $d_{\Omega_{\alpha,\beta}}(0, n\sqrt{2} e_{d}) \geq 2n (\alpha - \beta)$. Plotting this in \eqref{C geo} yields $C_{\Omega_{\alpha,\beta}}(e_{d}) \leq \frac{1}{\sqrt{2}(\alpha - \beta)}$. Taking $\beta$ small enough, and increasing $\alpha$ if needed, we can assume that $C_{\Omega_{\alpha,\beta}}(e_{d}) \leq \frac{1}{\kappa}$. Denoting $w_{\alpha,\beta}$ the speed of invasion in the domain $\Omega_{\alpha,\beta}$, Proposition \ref{geo} implies that $w_{\alpha,\beta}(e_{d}) \leq \frac{1}{\kappa}2\sqrt{f^{\prime}(0)}$. Hence, $c^{\star}_{\alpha,\beta}(e_{x})>w_{\alpha,\beta}(e_{d})$ when $\alpha$ is close enough to $1$ and $\beta$ close enough to $0$. This yields the result.

\end{proof}

\subsubsection{Combustion and monostable case}

Now, we answer Question \ref{question} in the case where $f$ is a combustion or a monostable nonlinearity. We do it for $f$ combustion first, and then we explain how this yields the result for monostable nonlinearities.

\begin{prop}\label{dim2}
Let $f$ be a combustion nonlinearity, i.e, it satisfies \eqref{combustion}. Then, there are $L>0$ and a family of smooth $L$-periodic domains $\left( \Omega_{\alpha} \right)_{\alpha \in (0,1)}$ such that $w_{\alpha}(e_{x}) \geq K$,  where $K>0$ is independent of $\alpha$, and $w_{\alpha}(e_{y}) \to 0 $ as $\alpha$ goes to $0$. 
\end{prop}
If $\alpha>0$ is chosen small enough so that $w_{\alpha}(e_{y})<K$, we see that $w_{\alpha}$ can not be constant, and then Proposition \ref{prop} implies that $c^{\star} \not\equiv w$ on $\Omega_{\alpha}$ for $\alpha$ small. This answers Question \ref{question} and proves Theorem \ref{thwc} when $f$ is a combustion nonlinearity.

Before turning to the proof of Proposition \ref{dim2}, we state the following technical lemma. We mention that Rossi and the author prove a similar (more general) lemma in \cite{DR}. We recall that we denote $B_{R}$ the ball of radius $R$ and of center $0$.

\begin{lemma}\label{tech}
Let $f$ be a combustion nonlinearity \eqref{combustion} independent of $x$. Then, there are $R,c>0$ and $\phi \in W^{2,\infty}(\mathbb{R}^{2})$,  $\phi>0$ in $B_{R}$ and $\phi =0$ on $\partial B_{R}$ such that, on $B_{R}$ we have:
$$
\Delta \phi + c\partial_{x}\phi +f(\phi) \geq 0.
$$ 
\end{lemma}

\begin{proof}

We build $\phi$ to be radial. We set $\phi(x) := h(\vert x \vert)$. Now, take $R_1, R_2 , R_3 >0$ to be chosen after, such that $R_1 < R_2 < R_3$. We denote  $\tilde{c} := c +\frac{1}{R_1}$. Let $C \in (\theta ,1)$, and $\alpha, \beta >0$. We define $h$ as follows:

\begin{equation}
\left\{
\begin{array}{lllll}
h(r) &=& C,& \quad &r \in [ 0, R_1 ],  \\
h(r) &=& -\frac{\alpha}{2}(r-R_1)^{2}+C,& \quad &r \in [ R_1, R_2 ],  \\
h(r) &=&\beta(e^{-\tilde{c}(r-R_3)}-1),& \quad &r \in [ R_2, R_3 ].
\end{array}
\right.
\end{equation}
Let us see that we can choose $R_1, R_2 , R_3 , c,\alpha, \beta , C $ such that 
\begin{equation}\label{conditions h}
\left\{
\begin{array}{lllrr}
h \in W^{2,\infty}(\mathbb{R}^{+}), \\
h(R_2) = K , \ \text{ where } K \in (\theta , C) \text{ will be chosen after}, \\
h^{\prime \prime}(r) + \tilde{c}h^{\prime}(r) +f(h(r))\geq0, \ \text{ for } r \geq 0.
\end{array}
\right.
\end{equation}

The existence of such a function proves our result, indeed  
\begin{equation*}
\begin{array}{lll}
\Delta \phi + c\partial_{x}\phi +f(\phi) &\geq h^{\prime\prime} + (c+\frac{1}{r})h^{\prime} + f(h) \\
&\geq  h^{\prime\prime} + (c+\frac{1}{R_1})h^{\prime} + f(h).
\end{array}
\end{equation*}
We used the fact that $h$ is non-increasing and $h^{\prime}(r)=0$ if $r\in [0,R_1]$ here.

Let us define
$$
F := \inf_{s\in (K,C)} f(s) >0.
$$
Because $h(R_2) =K$, we can bound from behind $f(h(r))$ by $F$ when $r \in [R_1,R_2]$ and by $0$ elsewhere. Some easy computations show that \eqref{conditions h} boils down to verify the following algebraic system:

\begin{equation}\label{algebraic}
\left\{
\begin{array}{llcc}
\beta(e^{\tilde{c}(R_3-R_2)}-1) = K, \\
\frac{\alpha}{2}(R_2 - R_1)^{2} = C-K, \\
\alpha(R_2 - R_1) = \beta \tilde{c}e^{\tilde{c}(R_3-R_2)}, \\
F \geq \alpha (1+\tilde{c}(R_2 - R_1)).
\end{array}
\right.
\end{equation}

Up to some computations, it is easy to see that \eqref{algebraic} admits positive solutions, for instance:

\begin{equation*}
\left\{
\begin{array}{llcc}
\alpha = \frac{F}{1+\frac{C-K}{2K}},   \\
c = \frac{1}{8}\frac{\sqrt{2\alpha(C-K)}}{K},\\
\beta = \frac{\sqrt{\alpha(C-K)}}{2c\sqrt{2}}-K,  \\
R_1 = \frac{1}{c}, \\
R_2 = \sqrt{\frac{2(C-K)}{\alpha}}+R_1,  \\
R_3 = \frac{1}{2c} \ln(1+\frac{K}{\beta})+R_2.
\end{array}
\right.
\end{equation*}
Hence, $\phi(x) := h(\vert x\vert)$ satisfies the lemma with $R := R_3$.

\end{proof}
Now, we use this technical lemma to prove Proposition \ref{dim2}. 

\begin{proof}[Proof of Proposition \ref{dim2}]

\emph{Step 1: Construction of the domain.} \\
Let $R>0$ be large enough, so that we can apply Lemma \ref{tech}. We build a family of $3R$-periodic domains, let $\alpha \in (0,1)$, $\varepsilon \in [0, \frac{\alpha R}{2}]$ and define
$$
\tilde{K}_{\alpha}^{\varepsilon} := \left\{ (x,y) \in \mathbb{R}^{2}\ \text{ such that } \ \alpha x + R +\varepsilon \leq y \leq \alpha x + (1+\alpha)R -\varepsilon, \ y \in [R , 2R]   \right\} .
$$
Now, let $K_{\alpha}$ be a smooth connected compact set such that
$$
\tilde{K}_{\alpha}^{\frac{\alpha R}{4}} \subset K_{\alpha} \subset   \tilde{K}_{\alpha}^{0}.
$$
We define $\Omega_{\alpha}$ to be a smooth $3R$-periodic domain as follows:
$$
\Omega_{\alpha}^{c} := \bigcup_{k \in \mathbb{Z}^{2}} (K_{\alpha}  +3R k ).
$$
Observe that, if $k ,l \in \mathbb{Z}^{2}$ are such that $k \neq l$, then $(K_{\alpha}  +3R k )\cap (K_{\alpha}  +3R l ) = \emptyset$.
\\
\\
\\
\includegraphics[scale = 0.7]{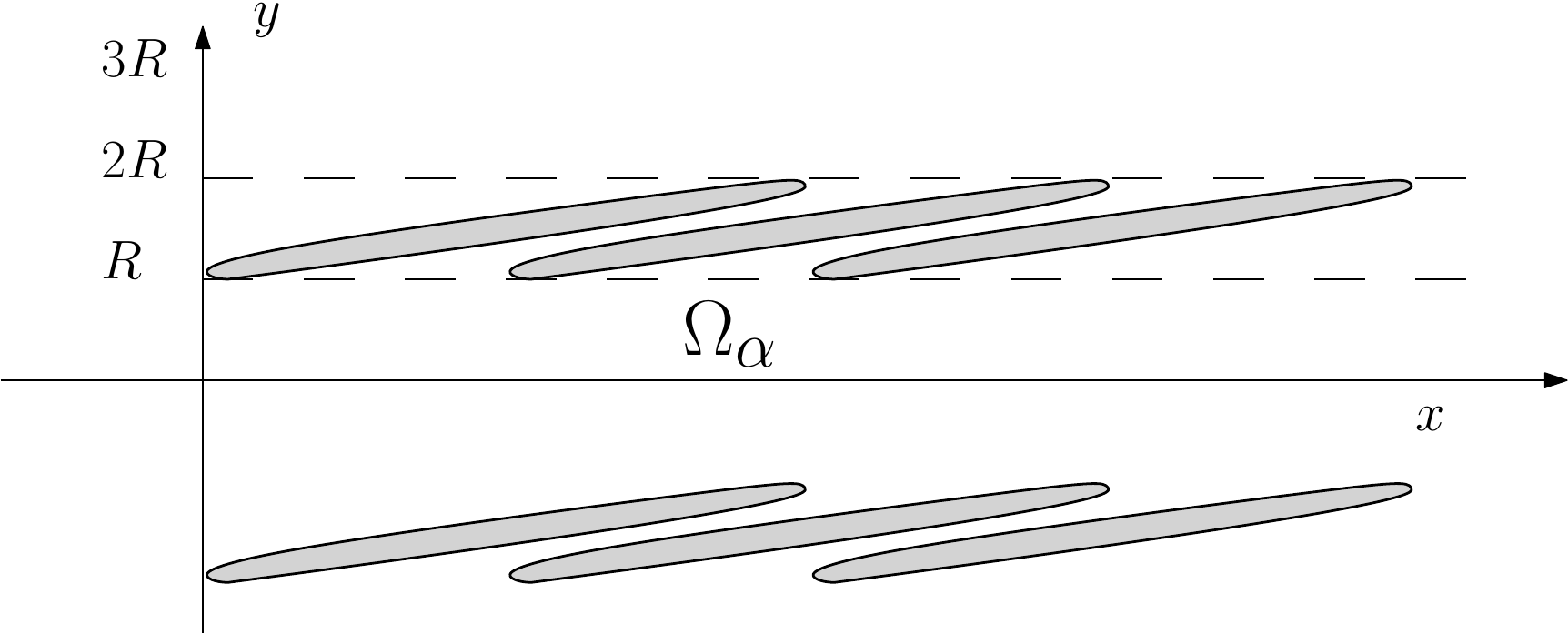}
\\
\\
\emph{Step 2 : Lower bound on $w_{\alpha}(e_{x})$.} \\
For $\alpha>0$ given, we can consider the invasion speed $w_{\alpha}(e_x)$ for \eqref{PeriodicFKPP} set on the smooth periodic domain $\Omega_{\alpha}$ (thanks to Theorem \ref{FG}). Let us show that there is $K>0$ independent of $\alpha$ such that $w_{\alpha}(e_{x}) \geq K$.

Because of the choice of $R$, we can apply Lemma \ref{tech} to find $c>0$ and $\phi \in W^{2,\infty}(B_{R})$, $\phi >0$ on $B_{R}$ and $\phi = 0$ on $\partial B_{R}$ such that $\Delta \phi + c\partial_{x}\phi +f(\phi) \geq 0$. Now, we define

\begin{equation*}
v(t,x,y) := 
\left\{
\begin{array}{lll}
\phi(x-ct,y) &\text{ if }  (x,y) \in B_{R}(cte_x),  \\
0 &\text{ elsewhere. }
\end{array}
\right.
\end{equation*}
Then, the support of $v(t,\cdot,\cdot)$ never intersects the boundary of $\Omega_{\alpha}$, and because 
$$
\partial_{t}v - \Delta v -f(v) = -c\partial_{x}\phi - \Delta\phi -f(\phi) \leq 0 \text{ for } (x,y) \in \text{supp}\, v(t,\cdot, \cdot),
$$
we have that $v$ is a non-negative compactly supported generalized subsolution of \eqref{PeriodicFKPP}.

Now, take $u_{0}$ a compactly supported initial datum such that $u_{0}(x,y) \geq \phi(x,y)$ and such that the solution arising from it, name it $u(t,x,y)$, converges to $1$ (as we mentioned earlier, such initial datum always exists, see \cite{DR}). The parabolic comparison principle yields
$$
u(t,x,y) \geq \phi(x-ct,y), \quad \forall t \geq 0, \ \forall(x,y) \in \Omega_{\alpha}.
$$
By definition of $w_{\alpha}(e_{x})$, this implies that $w_{\alpha}(e_{x}) \geq c$, where $c$ is given by Lemma \ref{tech} and is independent of $\alpha$. This concludes this step with $K := c$.

\emph{Step 3. Upper bound bound on $w_{\alpha}(e_{y})$.} \\
We now show that $w_{\alpha}(e_{y}) \to 0 $ as $\alpha$ goes to $0$. To do so, we first apply Proposition \ref{geo}, to get
$$
w_{\alpha}(e_{y}) \leq 2C_{\Omega_{\alpha}}(e_{y})\sqrt{\max_{u\in [0,1]}\frac{f(u)}{u}}.
$$
Let us estimate $C_{\Omega_{\alpha}}(e_{y})$. If we take $n \in \mathbb{N}$, we see that, if $\alpha$ is small enough, $d_{\Omega_{\alpha}}(0,4Rne_{y} ) \geq 2Rn \sqrt{1+\left(1-\frac{1}{\alpha}\right)^{2}}$. If $\alpha$ is small enough, $C_{\Omega_{\alpha}}(e_{y})\leq 3\alpha$. Then
$$
w_{\alpha}(e_{y}) \leq 6\alpha\sqrt{\max_{u\in [0,1]}\frac{f(u)}{u}} \underset{\alpha \to 0}{\longrightarrow} 0,
$$
hence the result.

\end{proof}

Now, Proposition \ref{dim2} is proved, and answers Question \ref{question} in the combustion case : in $\Omega_{\alpha}$, $c^{\star} \not\equiv w$, for $\alpha>0$ small enough.

Let us now explain how this also answers Question \ref{question} in the monostable case. Take $f$ to be monostable nonlinearity and let $\underline{f}$ be a combustion nonlinearity and let $\overline{f}$ be a KPP nonlinearity, both independent of $x$, such that 
$$
\underline{f} \leq f \leq \overline{f}.
$$
Let $\overline{w}_{\alpha}$, $w_{\alpha}$, $\underline{w}_{\alpha}$ be the invasion speed for the problem \eqref{PeriodicFKPP} with nonlinearity $\overline{f}, \ f, \ \underline{f}$ respectively. Then, by comparison, we have 
\begin{equation}\label{encadrement}
\forall e \in \mathbb{S}^{1}, \quad \underline{w}_{\alpha}(e) \leq w_{\alpha}(e) \leq \overline{w}_{\alpha}(e).
\end{equation}
Now, we can apply Proposition \ref{tech} to find $c>0$ and $\phi \in W^{2,\infty}(\mathbb{R}^{2})$,  $\phi>0$ in $B_{R}$ and $\phi =0$ on $\partial B_{R}$ such that, on $B_{R}$ we have:
$$
\Delta \phi + c\partial_{x}\phi +\underline{f}(\phi) \geq 0.
$$ 
Then, consider the domain $\Omega_{\alpha}$ built in the proof of Proposition \ref{dim2}, but with this new $R>0$.

On this domain, we have a lower bound on $\underline{w}_{\alpha}(e_{x})$ independent of $\alpha$. Moreover, we can show that $\overline{w}_{\alpha}(e_{y})$ goes to zero as $\alpha$ goes to $0$, as in the proof of Proposition \ref{dim2}.

Hence, \eqref{encadrement} yields that there is $K>0$ independent of $\alpha$ such that $w_{\alpha}(e_{x}) \geq K$, and $w_{\alpha}(e_{y}) \to 0 $ as $\alpha$ goes to $0$. This means that Proposition \ref{dim2} still holds if $f$ is monostable, hence this answers Question \ref{question} in the monostable case and concludes the proof of Theorem \ref{thwc}.

\subsection{Symmetries of the domain and relation with $c^{\star}$ and $w$}\label{equality}

This section is dedicated to the proof of Theorem \ref{symmetry}. As we mentioned earlier, even in a domain $\Omega$ where $c^{\star} \not\equiv w$, the Freidlin-Gartner formula yields that any direction $e\in \mathbb{S}^{N-1}$ minimizing $c^{\star}$ satisfies the equality $c^{\star}(e) = w(e)$. Theorem \ref{symmetry} gives a geometrical condition that ensure the existence of directions where $c^{\star}$ and $w$ coincide. To prove it, we first start to state the following lemma:

\begin{lemma}\label{tech sym}
Let $c^{\star}$ and $w$ be respectively the critical speed of fronts and the speed of invasion for \eqref{PeriodicFKPP} with the nonlinearity $f$ satisfying \eqref{regularity}, \eqref{hypf} and such that $u \mapsto \frac{f(u)}{u}$ is non-increasing. For any $k \in \mathbb{N}$ and $e \in \mathbb{S}^{N-1}, \ (\xi_{i})_{i \in \llbracket 1 , k \rrbracket }~\in~(\mathbb{S}^{N-1})^{k}$ such that 
$$
e \in \left\{ x \in \mathbb{R}^{N} \ :  \ x = \sum_{i=1}^{k} \lambda_{i}\xi_{i}, \ \lambda_{i}\geq 0  \right\},
$$
the following holds:

$$
c^{\star}(e) \leq \max_{i \in \llbracket 1 , k \rrbracket }  \frac{c^{\star}(\xi_{i})}{e\cdot \xi_{i} } .
$$

\end{lemma}

\begin{proof}[ Proof of Lemma \ref{tech sym}]
For $i\in \llbracket 1 , k \rrbracket$, we denote $ \phi_{\xi_{i}}(t,x)$ a pulsating traveling fronts solution of \eqref{eqgen} in the direction $ \xi_{i}$  with critical speed $ c^{\star}(\xi_{i})$ respectively. Denote 
$$
v(t,x) :=  \sum_{i = 1}^{k}\phi_{\xi_{i}}(t,x).
$$
Now, the hypotheses on $f$ imply that $f(v) \leq \sum_{i = 1}^{k}f(\phi_{\xi_{i}})$, and then $v$ is a supersolution of \eqref{eqgen}.

Now, for $\varepsilon>0$, let $f_{\varepsilon}$ be a combustion nonlinearity satisfying the following:
\begin{equation*}
\left\{
\begin{array}{llc}
0\leq f_{\varepsilon}(x,u) \leq f(x,u),    \quad &\forall u \in [0,1], \ \forall x\in \Omega,      \\
f_{\varepsilon}(x,u) = f(x,u), \quad &\forall u \in [0,1-2\varepsilon], \ \forall x\in \Omega,  \\
f_{\varepsilon}(x,u) =0, \quad &\forall u \in [-\varepsilon , 0], \ \forall x \in \Omega, \\
f_{\varepsilon}(x,1-\varepsilon)  = 0 , \quad &\forall x \in \Omega.

\end{array}
\right.
\end{equation*}
Let $\phi_{e}^{\varepsilon}$ be a pulsating traveling fronts connecting $1-\varepsilon$ to $-\varepsilon$, solution of \eqref{eqgen} with the combustion nonlinearity $f_{\varepsilon}$, in the direction $e$ with critical speed $c^{\star}_{\varepsilon}(e)$. 

Up to some translation in time, we can assume that $\phi_{e}^{\varepsilon}(0,x)<0$ if $x\cdot e >0$ and $\forall i \in \llbracket 1, k \rrbracket$, \ $\phi_{\xi_{i}}(0,x) \geq 1-\varepsilon$ if $x\cdot \xi_{i} <0$. 
Thanks to the hypotheses, we can write $e = \sum_{i=1}^{k} \lambda_{i}\xi_{i}$, with $\lambda_{i}\geq 0, \ \forall i \in \llbracket 1,k \rrbracket $. Hence, if $x\in \Omega$ is such that  $x\cdot e <0$, then there is at least one of the $\xi_{i}$ such that $x\cdot \xi_{i}<0$. Hence, $v(0,x) > 1-\varepsilon$ if $x\cdot e <0$. If $x\cdot e \geq0$, we have $v(0,x) >0\geq \phi_{e}^{\varepsilon}(0,x)$. Hence

$$
v(0,x) \geq \phi_{e}^{\varepsilon}(0,x), \quad  \forall x\in \Omega.
$$ 
Because $f_{\varepsilon} \leq f$, the parabolic comparison principle yields
\begin{equation}\label{v phi}
v(t,x) \geq \phi_{e}^{\varepsilon}(t,x) , \quad \forall t \geq 0,\ \forall x \in \Omega.
\end{equation}
Now, if we take $\overline{c}>\max_{i \in \llbracket 1 , N \rrbracket }  \frac{c^{\star}(\xi_{i})}{e\cdot \xi_{i} }$, we have that $v(t,\overline{c}te) \to 0$ as $t$ goes to $+\infty$. It then follows from \eqref{v phi} that $c_{\varepsilon}^{\star}(e) \leq \max_{i \in \llbracket 1 , N \rrbracket }  \frac{c^{\star}(\xi_{i})}{e\cdot \xi_{i} }$. 
Now, it is classical that  $c^{\star}_{\varepsilon}(e) \to c^{\star}(e)$ as $\varepsilon$ goes to $0$ (see, for exemple, \cite[Proposition 2.6]{R1}). Taking the limit $\varepsilon \to 0$  then yields the result.

\end{proof}

\begin{req}
Lemma \ref{tech sym} yields a very strong geometrical condition on $c^{\star}$, and prevents it to be any arbitrary function. Consider 
$$
\mathcal{C} := \left\{  r(\xi) \xi \in \mathbb{R}^{2} \ : \ r(\xi) \in [0,c^{\star}(\xi)]   \right\}.
$$
In the case of equation \eqref{PeriodicFKPP} with $\Omega = \mathbb{R}^{N}$, $c^{\star}$ is constant and then $\mathcal{C}$ is a ball. In general, it is not clear what ``shapes" $\mathcal{C}$ can adopt. Lemma \ref{tech sym} prevents it to be some natural candidates, for instance, $\mathcal{C}$ can not be an ellipse with eccentricity larger than $\frac{1}{\sqrt{2}}$.  We recall that an ellipse of equation $\frac{x^{2}}{a^{2}} + \frac{y^{2}}{b^{2}}=1$, with $a>b$, has eccentricity $\sqrt{1-\frac{b^{2}}{a^{2}}}$. 
\end{req}
Now, we prove Theorem \ref{symmetry}.

\begin{proof}[Proof of Theorem \ref{symmetry}.]

We first consider the case $N \geq 3$. Let $e\in \mathbb{S}^{N-1}$ and $R$ be the rotation of angle $\pi$ around the axis directed by $e$ such that $R\Omega = \Omega$. Take $\xi_{1}$ be such that $w(e) = \frac{c^{\star}(\xi_{1})}{\xi_{1} \cdot e}$, where $\xi_{1}\cdot e >0$. If $\xi_{1}=e$, then $w(e) = c^{\star}(e)$ and we are done. If not, define $\xi_{2} := R\xi_{1}$. Then, $e$ is in the positive cone generated by $\xi_{1}$ and $\xi_{2}$. Because $R\Omega = \Omega$, it is easy to see that $c^{\star}(\xi_{1})=c^{\star}(\xi_{2})$. Indeed, if $\phi(t,x)$ is a pulsating traveling front solution of \eqref{PeriodicFKPP} in the direction $\xi$ with speed $c^{\star}(\xi)$, then $\phi(t,Rx)$ is a pulsating traveling front solution of \eqref{PeriodicFKPP} in the direction $R\xi$ with speed $c^{\star}(\xi)$. Hence, Lemma \ref{tech sym} implies that  
$$
c^{\star}(e) \leq   \frac{c^{\star}(\xi_{1})}{\xi_{1}\cdot e } = w(e) .
$$
Because $w(e) \leq c^{\star}(e)$ (thanks to \eqref{FGformula}), we get
$$
c^{\star}(e) = w(e),
$$
hence the result. If $N=2$, the proof works identically.

\end{proof}

Let us emphasize that, if one consider the general equation \eqref{eqgen}, then Theorem \ref{symmetry} still holds if  we add the following hypothesis on the coefficients: in the case $N\geq 3$,
$$
A(Rx) = RA(x)R^{\star}, \ q(R x) = Rq(x), \ f(R x,\cdot) = f(x,\cdot),
$$
where  $R$ is the rotation of angle $\pi$ around the axis directed by $e$ that satisfies $R\Omega = \Omega$.  In the case $N=2$, we should add:
$$
 A(Sx) = SA(x)S^{\star}, \ q(S x) = Sq(x), \ f(S x,\cdot) = f(x,\cdot),
$$
where  $S$ is the symmetry about the axis directed by $e$ that satisfies $S\Omega = \Omega$.

\section*{Appendix}

\subsection*{Proof of Lemma \ref{lemma1}}

As we mentioned, Lemma \ref{lemma1} is the natural extension of \cite[Lemma 2.1]{R1}, in the case of a periodic domain.

\begin{proof}

Let $u$ be taken as in the lemma. We denote 
$$
h\::=\liminf_{\delta\to+\infty}\inf_{\substack{t<0\\ x \cdot e< \gamma t-\delta \\ x\in\Omega }} u(t,x)
$$

Assume that, by contradiction,  $h\in (S,1)$. We can find two  sequences $(x_{n})_{n}\in\Omega^{\mathbb{N}}$, $(t_{n})_{n}\in (-\infty,0)^{\mathbb{N}}$ such that $x_{n}\cdot e- \gamma t_{n} \to-\infty$ and $u(t_{n},x_{n})\to h$ as $n$ goes to $+\infty$. Let us define $k_{n}\in\mathbb{Z}^{N},\: z_{n}\in[0,1)^{N}$ so that $x_{n}=k_{n}+z_{n}$. Up to extraction, we assume that $z_{n}\to z$ as $n$ goes to $+\infty$, for some $z\in [0,1]^{N}$. Consider the sequence of translated functions:
$$
u_{n}:=u(\cdot+t_{n},\cdot+k_{n}).
$$
These functions are supersolutions of \eqref{eqgen}, by periodicity of the domain. As before, we can use the usual parabolic estimates to get local uniform convergence of the sequence $(u_{n})_{n}$ to a function $u_{\infty}$ supersolution of \eqref{eqgen}. Moreover, we have 
\begin{equation}\label{dem lemme 1}
u_{\infty}(0,z)=h\leq u_{\infty}(t,x), \quad \forall t\leq0 ,\ \forall x \in \Omega.
\end{equation}
Indeed, for $t\leq0,  x \in \Omega$, we have
\begin{equation}\label{tech n}
u_{n}(t,x) = u(t+t_{n},x+k_{n}) \geq \inf_{\substack{\tau <0   \\  y\cdot e -\gamma \tau \leq \tilde{\delta}_{n}}}u(\tau , y),
\end{equation}
where $\tilde{\delta}_{n} := x\cdot e - \gamma t -z_{n}\cdot e + x_{n}\cdot e -\gamma t_{n}$ goes to $-\infty$ as $n$ goes to $+\infty$. Hence, passing to the limit $n\to + \infty$ in \eqref{tech n} yields \eqref{dem lemme 1}. Because $f\geq 0$, it follows from the parabolic maximum principle and Hopf principle that $u_{\infty}$ is actually equal to $h$ if $ t \leq 0$ and $x\in \Omega$. This implies that $f(x,h)= 0$, which is in contradiction with the fact that $h \in (S,1)$ together with hypothesis \eqref{hypf}, hence the result.

\end{proof}

\subsection*{Proof of Lemma \ref{lemma2}}
We now turn to the proof of Lemma \ref{lemma2}.  Again, it is the natural extension of  \cite[Lemma 2.2]{R1} to the case of a periodic domain.

\begin{proof}

Let us define $\bar{u}_{\varepsilon} :=\bar{u}+\varepsilon$, where $\varepsilon>0$. The hypotheses on $\overline{u}$ yield that there is $\delta>0$ such that $\overline{u}_{\varepsilon}(t,x)\geq 1 + \frac{\varepsilon}{2}$ if $t<0$ and $x\cdot e < \gamma t -\delta$, $x\in \Omega$. The hypotheses on $\underline{u}$ gives us that there is $L>0$ such that $\underline{u}(t,x) \leq \varepsilon$ if $t<0$ and $x\cdot e \geq (\gamma+\eta)t +L$. Hence, there is $T_{\varepsilon}\leq 0$, such that $\overline{u}_{\varepsilon}(t,x) > \underline{u}(t,x)$ for $t < T_{\varepsilon}$, for all $x\in \Omega$. Indeed, if $t$ is negative enough, we have $\eta t + L < -\delta$, hence we can take $T_{\varepsilon} := \frac{-\delta-L}{\eta}$.

In order to prove the result, we shall argue by contradiction. Hence, we will assume that there is $\varepsilon_{0}>0$ such that:

\begin{equation}\label{lemme tau}
\forall\varepsilon\in(0,\varepsilon_{0}),\  \exists\tau\in(T_{\varepsilon},\:0),\ \exists x_{\tau}\in\Omega\ \ \text{such that} \ \bar{u}_{\varepsilon}(\tau,x_{\tau})<\underline{u}(\tau,x_{\tau}).
\end{equation}
Indeed, if \eqref{lemme tau} does not hold, our result follows by letting $\varepsilon \to 0$. Now, we define $t_{\varepsilon}\in[T_{\varepsilon},0)$ to be the infimum of all the $\tau$ such that \eqref{lemme tau} holds true. Hence
$$
\overline{u}_{\varepsilon}(t,x)\geq\underline{u}(t,x),\ \ \forall t\leq t_{\varepsilon}, \ \forall x \in \Omega,
$$
and by continuity we have
$$
\inf_{x\in\Omega}\big(\bar{u}_{\varepsilon}-\underline{u}\big)(t_{\varepsilon},x)=0.
$$
Thanks to the hypotheses, we can find $\rho_{\varepsilon}\in \mathbb{R}$
such that
$$
\inf_{x\cdot e=\rho_{\varepsilon}}\big(\bar{u}_{\varepsilon}-\underline{u}\big)(t_{\varepsilon},x)=0.
$$
Depending on the behavior of $\rho_{\varepsilon}$, we now consider three cases.

\emph{First case : $(\rho_{\varepsilon})_{\varepsilon\in(0,\varepsilon_{0})}$ is
bounded.} \\ We can find a sequence of points $(x_{\varepsilon})_{\varepsilon\in(0,\varepsilon_{0})}$, with $x_{\epsilon} \in \Omega$ such that 

\[
x_{\varepsilon}\cdot e=\rho_{\varepsilon}\ \text{ and } \ \bar{u}_{\varepsilon}(t_{\varepsilon},x_{\varepsilon})-\underline{u}(t_{\varepsilon},x_{\varepsilon})<\varepsilon.
\]
We define $k_{\varepsilon}\in\mathbb{Z}^{N},\: y_{\varepsilon}\in [0,1)^{N}$ to be such that $x_{\varepsilon}=k_{\varepsilon}+y_{\varepsilon}$ . Up to extraction,we can find $y \in [0,1)^{N}$ such that $y_{\varepsilon}\to y$ as $\varepsilon$ goes to $0$.

We now consider the translated functions $\overline{u}_{\varepsilon}(t+t_{\varepsilon},x+k_{\varepsilon}),\ \underline{u}(t+t_{\varepsilon},x+k_{\varepsilon})$. Using parabolic estimates and extracting, these functions converge locally uniformly as $\varepsilon$ goes to $0$ to $\overline{u}_{\infty},\:\underline{u}_{\infty}$, a supersolution and a subsolution respectively of \eqref{eqgen}.

Moreover, $\overline{u}_{\infty},\:\underline{u}_{\infty}$ satisfy:

$$
\overline{u}_{\infty}(0,y)=\underline{u}_{\infty}(0,y) \ \text{ and } \  \overline{u}_{\infty}(t,x)\geq\underline{u}_{\infty}(t,x)\ \text{ for } \ t\leq 0, \ x\in \Omega.
$$
Hence, the strong comparison principle and Hopf lemma (see \cite[chapter 3]{PW})  imply that $\overline{u}_{\infty}=\underline{u}_{\infty}$ for $t\leq 0$. But the boundedness of $x_{\varepsilon}\cdot e=\rho_{\varepsilon}$ implies that we still have

$$
\liminf_{\delta\to+\infty}\inf_{\substack{t<0\\
x\cdot e<\gamma t\\
x\in\Omega
}
-\delta}\overline{u}_{\infty}(t,x)\geq 1.
$$
However, the hypotheses on $\underline{u}$ yields that there is $K \in \mathbb{R}$ such that
$$
\underline{u}_{\infty}(t,x)\leq \frac{1}{2}, \quad \forall t <0, \ \forall x \in \Omega \text{ such that } x\cdot e \geq (\gamma +\eta)t +K.
$$
Taking $t<0$ small enough yields a contradiction.

\emph{Second case : $\inf_{\varepsilon\in(0,\varepsilon_{0})}\rho_{\varepsilon}=-\infty$.} \\ Let us take $\varepsilon$ be such that $-\rho_{\varepsilon}$ is large
enough to have
$$
\inf_{\substack{\substack{t<0\\
x\cdot e-\gamma t<\rho_{\varepsilon}
}
}
}\overline{u}(t,x)>S.
$$
Because $f(x,\cdot)$ is decreasing in $(S,1)$, we have that $\overline{u}_{\varepsilon} = \overline{u}+\varepsilon$ is supersolution of \eqref{eqgen} for $\{(t,x) \in \mathbb{R}\times \Omega \text{ such that } x\cdot e-\gamma t<\rho_{\varepsilon} \}$.

We can find a sequence $(x_{n})_{n} \in \Omega^{\mathbb{N}}$ such that $x_{n} \cdot e =0$ and
$$
\lim_{n\to+\infty}\big(\overline{u}_{\varepsilon}-\underline{u}\big)(t_{\varepsilon},\rho_{\varepsilon}e+x_{n})=0.
$$
We write as before  $x_{n}=k_{n}+y_{n}$, where $k_{n}\in \mathbb{Z}^{N}$ and $y_{n}\in \left[ 0 , 1\right)^{N}$, and up to extraction we can find $y \in \left[ 0 , 1\right]^{N}$ such that $y_{n}\to y$ as $n$ goes to $+\infty$.

We define $\overline{u}^{\varepsilon}_{n}(t,x) := \overline{u}_{\varepsilon}(t,x+k_{n})$ and $\underline{u}_{n}(t,x) := \underline{u}(t,x+k_{n})$. Observe that $\overline{u}^{\varepsilon}_{n}$ is supersolution in $\{(t,x) \in \mathbb{R}\times \Omega \text{ such that } x\cdot e-\gamma t<\rho_{\varepsilon}-1 \}$. Again, using parabolic estimates and extracting as $n$ goes to $+ \infty$, we get two functions $\overline{u}_{\infty}^{\varepsilon}$ and $\underline{u}_{\infty}$ that are respectively supersolution and subsolution of \eqref{eqgen} on the same set. 
Moreover, they satisfy $\overline{u}_{\infty}^{\varepsilon}(t_{\varepsilon},\rho_{\varepsilon}e+y)=\underline{u}_{\infty}(t_{\varepsilon},\rho_{\varepsilon}e+y)$ : we have a contact point.

Observe that $(t_{\varepsilon},\rho_{\varepsilon}e+y)$ is in $\{(t,x) \in \mathbb{R}\times \Omega \text{ such that } x\cdot e-\gamma t<\rho_{\varepsilon}-2 \}$. Hence, we can apply Hopf lemma (\cite[Theorem 6]{PW}) to the equation \eqref{eqgen} on $\{(t,x) \in \mathbb{R}\times \Omega \text{ such that } x\cdot e-\gamma t<\rho_{\varepsilon}-1 \}$ to get get that $(t_{\varepsilon},\rho_{\varepsilon}e+y)$ is not on a boundary point. Therefore, it is an interior contact point and the parabolic comparison principle yields that $\overline{u}_{\infty}^{\varepsilon}(t,x)=\underline{u}_{\infty}(t,x)$  on $\{(t,x) \in \mathbb{R}\times \Omega \text{ such that } x\cdot e-\gamma t<\rho_{\varepsilon}-1 \}$. But this is not possible, because the hypotheses on $\overline{u}$ imply that there is $\delta$  large enough so that $\overline{u}_{\infty}^{\varepsilon}(t,x)\geq 1+\frac{\varepsilon}{2}$ if $x\cdot e -\gamma t <-\delta$. Because $\underline{u}\leq 1$, we are led to a contradiction.

\emph{Third Case : $\sup_{\varepsilon\in(0,\varepsilon_{0})}\rho_{\varepsilon}=+\infty$.} \\ If we are in the case \eqref{compmono}, this can not happen because $\overline{u}_{\varepsilon}\geq 0$ and $\underline{u}(t_{\varepsilon},x) < 0$ if $x\cdot e$ is large enough.  Then, we are left to assume that $f$ satisfies \eqref{decroissance}
and $\underline{u}$ satisfies \eqref{compcomb}. In particular, we can take $\varepsilon$ small enough so that $\rho_{\varepsilon}$ is large enough to have $\underline{u}(t,x)\leq \theta$ on $\{(t,x) \in \mathbb{R}\times \Omega \text{ such that } x\cdot e-\gamma t>\rho_{\varepsilon} \}$, where $\theta$ is from \eqref{decroissance}. Hence, $\underline{u}_{\varepsilon} := \underline{u}-\varepsilon$ is a subsolution of \eqref{eqgen} on this set. Arguing as in the previous case, we get a contradiction, and hence the result.
\end{proof}

\noindent {\textbf{Acknowledgements:} The research leading to these results has received funding from the European Research Council under the European Union's Seventh Framework Program (FP/2007-2013) / ERC Grant Agreement n.321186 - ReaDi - Reaction-Diffusion Equations, Propagation and Modeling. The author wants to thank Luca Rossi for suggesting this problem and for interesting discussions.}

\end{document}